\documentclass[matprg]{svjour3}
\usepackage{graphics,graphicx,epsf,subfigure,epstopdf}
\usepackage{amsmath,amsxtra,amsfonts,amscd,amssymb,bm}
\usepackage{algorithm}
\usepackage{algpseudocode}
\usepackage[mathscr]{eucal}
\usepackage{color}
\numberwithin{equation}{section}

\usepackage{hyperref}

\def\Argmin{{\rm Argmin}}
\def\eqnok#1{(\ref{#1})}

\newcommand{\tsum}{\textstyle\sum}
\newcommand{\bbe}{\mathbb{E}}
\def\prob{\mathop{\rm Prob}}
\def\Prob{{\hbox{\rm Prob}}}
\def\ub{{\rm ub}}
\def\lb{{\rm lb}}

\def\bmA{{\bm A}}
\def\bmc{{\bm c}}
\def\bmB{{\bm B}}
\def\bmQ{{\bm Q}}
\def\bmb{{\bm b}}
\def\bmp{{\bm p}}

\def\pIter{{\bm q}}
\newcommand{\beq}{\begin{equation}}
\newcommand{\eeq}{\end{equation}}
\newcommand{\beqa}{\begin{eqnarray}}
\newcommand{\eeqa}{\end{eqnarray}}
\newcommand{\beqas}{\begin{eqnarray*}}
\newcommand{\eeqas}{\end{eqnarray*}}

\newtheorem{assumption}{Assumption}

\newcommand{\bbr}{\Bbb{R}}
\newcommand{\nn}{\nonumber}
\def\cV{{\cal V}}

\def\vgap{\vspace*{.1in}}

\title{Complexity of Stochastic Dual Dynamic Programming
\thanks{
This research was partially supported by the NSF grant 1953199 and NIFA grant 2020-67021-31526.}
}
\author{
     Guanghui Lan 
    \thanks{H. Milton Stewart School of Industrial and Systems
    Engineering, Georgia Institute of Technology, Atlanta, GA, 30332.
    (email: {\tt george.lan@isye.gatech.edu}).}
}

\date{Submitted: December 16, 2019; Revised: May 27, 2020, July 31, 2020, November 2, 2021. \\
\vspace{0.1in}\\
{\sl Dedicated to Professor Alexander Shapiro on the occasion of his 70th birthday for his profound contributions to stochastic optimization.}}

\begin{document}

\maketitle

\begin{abstract}
Stochastic dual dynamic programming is
a cutting plane type algorithm for multi-stage stochastic optimization originated about 30 years ago. 
In spite of its popularity in practice, there does not
exist any analysis on the convergence rates of this method. In
this paper, we first establish the number of iterations, i.e., iteration complexity, required by  
a basic dual dynamic programming method for solving single-scenario multi-stage optimization problems,
by introducing novel mathematical tools including the saturation of search points.
We then refine these basic tools and   
establish the iteration complexity for
an explorative dual dynamic programing method proposed herein and
the classic stochastic dual dynamic programming method for solving more
general multi-stage stochastic optimization problems under the standard stage-wise
independence assumption. Our results indicate that the complexity of some deterministic variants of these methods 
mildly increases with the number of stages $T$, in fact linearly dependent on $T$
for discounted problems. Therefore, they are efficient for strategic decision making which
involves a large number of stages, 
but with a relatively small number of 
decision variables in each stage. Without explicitly
discretizing the state and action spaces, these methods might  also be 
pertinent to the related reinforcement learning and stochastic control areas.

\end{abstract}

\section{Introduction} \label{sec_intro}

In this paper, we are interested in solving the following stochastic
dynamic optimization problem
\begin{equation} \label{multstage}
\min_{
x_1 \in X_1
}
H_1(x_1,\bmc_1)
+ \lambda \bbe\left[
\min_{ x_2 \in X_2(x_1)} H_2(x_2, \bmc_2)
+\lambda \bbe\big [
\cdots +
\lambda \bbe[
\min_{
x_T \in X_T(x_{T-1})
} H_T(x_T, \bmc_T)
]
\big]
\right],
\end{equation}
with feasible sets $X_t$  given by
\begin{align}
X_1 &:= \left\{
x \in \bar X_1 \subseteq \bbr^{n_1}: \bmA_1 x_1 = \bmb_1, \Phi_1(x_1, \bmp_1) \le 0
\right\}, \\
X_t(x_{t-1}) &\equiv X_t(x_{t-1},\xi_t) \nn\\
&:= \left\{ x \in \bar X_t \subseteq \bbr^{n_t}: \bmA_t x = \bmB_t x_{t-1} + \bmb_t, \Phi_t(x, \bmp_t) \le \bmQ_t x_{t-1}\right\}.
\end{align}
Here $T$ denotes the number of stages, $H_t(\cdot,\bmc_t)$ are closed convex objective functions, $\bar X_t \subset \bbr^{n_t}$ 
are closed convex  sets, 
$\lambda \in (0,1]$ denotes
the discounting factor, $\bmA_t: \bbr^{n_t} \to \bbr^{m_t}, \bmB_t: \bbr^{n_{t-1}} \to \bbr^{m_t}$, and $\bmQ_t: \bbr^{n_{t-1}} \to \bbr^{p_t}$ are linear mappings, 
and $\Phi_{t,i}(\cdot, \bmp_t): \bbr^{n_t} \to  \bbr$, $i =1, \ldots, p_t$
are closed convex constraint functions. Moreover, $\xi_1 := (\bmA_1,\bmb_1,\bmB_1, \bmp_1,\bmc_1)$ is a given deterministic vector, and
$\xi_t := (\bmA_t,\bmb_t,\bmB_t, \bmQ_t, \bmp_t, \bmc_t)$, $t=2, \ldots, T$, are the random vectors
at stage $t$. 
In particular, if $H_t$ are affine, $X_t$ are polyhedral and $\Phi_t$ do not exist, then problem~\eqnok{multstage}
reduces to the well-known multi-stage stochastic linear programming problem (see, e.g., \cite{BirLou97,ShDeRu09}).
The incorporation of the nonlinear (but convex) objective functions $H_t$  and constraints $\Phi_t$
allows us to model a much wider class of problems. 

In spite of its wide applicability, multi-stage stochastic optimization remains highly challenging to solve.
As shown by Nemirovski and Shapiro~\cite{ShaNem04} and
Shapiro~\cite{sha06}, the number of scenarios of $\xi_t$, $t = 2, \ldots, T$, required 
to solve problem \eqnok{multstage} has to increase exponentially with $T$.
In particular, if the number of stages $T=3$, the total number of samples (a.k.a. scenarios)
should be of order ${\cal O}(1/\epsilon^4)$ in general. 
There exist many algorithms for solving multi-stage stochastic optimization problems (e.g.,~\cite{pereira1991multi,rockafellar1991scenarios,HigSen91-1}),
but quite often without guarantees provided on their rate of convergence.
More recently, Lan and Zhou~\cite{LanZhou17-1} developed
a dynamic stochastic approximation method for multi-stage
stochastic optimization by generalizing
stochastic gradient descent methods,
and show that this algorithm can achieve this optimal sampling and iteration complexity bound
for solving general multi-stage stochastic optimization problems with $T= 3$. 
The complexity of this method depends mildly on the problem dimensions, but increases exponentially 
with respect to $T$. As a result, this type of method is suggested for
solving some operational decision-making problems, which involve
a large number of decision variables but only a small number of stages.

In practice, we often encounter strategic decision making problems which span a long horizon and thus
require a large number of stages $T$.  In this situation, a crucial simplification that has been explored 
to solve problem~\eqnok{multstage} more efficiently is to assume the stage-wise independence.
In other words, we make the assumption that the random variables $\xi_t$, $t = 2, \ldots, T$, are
mutually independent of each other.
Under this assumption, we can write
problem~\eqnok{multstage} equivalently as
\begin{equation} \label{multstage1}
\begin{array}{ll}
\min_{x_1 \in X_1} \{ H_1(x_1,\bmc_1)+ \lambda V_{2}(x_1)  \},
\end{array}
\end{equation}
where the value factions $V_t$, $t = 2, \ldots, T$,  are recursively defined by
\begin{equation}\label{Defi_sto_V_m}
\begin{array}{lll}
 V_t(x_{t-1}) &:=& \bbe [\cV_t(x_{t-1},\xi_t)], \\
\cV_t(x_{t-1},\xi_t) &:= & \min_{x_t \in X_t(x_{t-1})} \{ H_t(x_t,\bmc_t)+ \lambda V_{t+1}(x_t)\},
 \end{array}
 \end{equation}
 and
\begin{equation}\label{Defi_sto_V1_m}
 \begin{array}{lll}
  V_{T+1}(x_{T}) = 0.
\end{array}
\end{equation}
Furthermore, as pointed out by Shapiro~\cite{Sha11}, one can generate a relatively small (i.e., $N_t$)
number of samples for each $\xi_t$ and define the so-called sample average approximation (SAA) problem
by replacing the expectation in~\eqnok{Defi_sto_V_m} with the average over the generated samples (see
Section~\ref{sec_DDDP} for more details).

Under the aforementioned stage-wise independence assumption, a widely-used method for solving
the SAA problem is the stochastic dual dynamic programming (SDDP) algorithm. SDDP is
an approximate cutting plane method, first presented by Birge~\cite{Birge85-1} and Pereira and Pinto \cite{pereira1991multi} and later studied
by Shapiro~\cite{Sha11}, Philpott et. al.~\cite{phil13-1}, Donohue and Birge~\cite{donohue2006abridged}, Hindsberger~\cite{hindsberger2014resa},
Kozm\'{i}k and Morton~\cite{kozmik2015evaluating}, 
Guigues~\cite{guigues2018inexact} and Zou et. al.~\cite{ZouAhmedSun19-1}, among many others.
SDDP has been applied to solve problems arising from many different fields such as hydro-thermal planning
\cite{Guigues14-1,ZouAhmedSun19-1} and bio-chemical process control \cite{BaoTong19-1}.
Each iteration of this algorithm contains two phases. In the forward phase, feasible solutions at each stage
will be generated starting from
the first stage based on the cutting plane models for the value functions built in the previous iteration.
Then in the backward phase, the cutting plane models for the value functions of each stage will be updated starting
from the last stage. 
While the cost per iteration of the SDDP method only linearly depends on the number of stages,
it remains unknown what is the number of iterations required by the SDDP method 
to achieve a certain accurate solution of problem~\eqnok{multstage1}. Existing
proofs of convergence  of  SDDP are based on the assumption that  the procedure passes 
through every possible scenario many times~\cite{Sha11,LinPhyi05-1,GirLecPhi15}. 
Of course when the number of scenarios, although finite, is astronomically large this is not very realistic.
In addition, such analysis does not reveal the dependence of the efficiency
of SDDP on various parameters, e.g., number of stages, target accuracy, Lipschitz constants, and
diameter of feasible sets etc.

 It is well-known that when the number of stages $T = 2$, SDDP reduces 
 to the classic Kelley's cutting plane method~\cite{Kelley60}.
As shown in Nesterov~\cite{Nest04}, 
the number of iterations required by Kelley's cutting plane method could  
depend exponentially on the dimension of the problem  even for a static optimization problem inevitably.
Therefore, this type of method is not recommended for solving large-scale
optimization problems. 
%
However, it turns out that the global
cutting plane models are critically important for multi-stage optimization especially if
the number of stages is large and one does not know the structure of optimal policies. 
In these cases we need to understand the efficiency of these cutting plane methods
in order to identify not only problem classes amenable for these techniques, but also possibly
to inspire new ideas to solve these problems more efficiently.
%

This paper intends to close the aforementioned gap in our understanding about cutting plane methods
for multi-stage stochastic optimization. Our main contributions mainly exist in the following several
aspects. Firstly, we start with a  dual dynamic programming (DDP) method
for solving dynamic convex optimization problem
with a single scenario. This simplification allows us to build a few essential
mathematical notions and tools for the analysis of cutting plane methods.
More specifically, we introduce the notion of saturated
and distinguishable search points. Using this notion,
we show that each iteration of DDP will either find a new saturated
and distinguishable search point, or compute an approximate 
solution for the original problem. As a consequence, we establish
the total number of iterations required by the DDP method
for solving the single-scenario problem. More specifically, we show that the
iteration complexity of DDP only mildly increases w.r.t. the number of stages $T$,
in fact linearly dependent on $T$ for many problems, especially
those with a discounting factor $\lambda < 1$.
The dependence of DDP on other problem parameters has also been thoroughly studied.
We also demonstrate that one can terminate DDP based on some
easily computable upper and lower bounds on the optimal value.

Secondly, motivated by the analysis of the DDP method, we propose
a new explorative dual dynamic programming (EDDP) for solving the SAA problem of
multi-stage stochastic optimization in \eqnok{multstage1}. 
When solving the SAA problem, we have to choose one
out of $N_t$ possible feasible solutions in the forward
phase, and each one of them corresponds
to a random realization of $\xi_t$. In EDDP, we choose
a feasible solution in an aggressive manner by selecting the most
distinguishable search point among the saturated ones in each stage.
As a result, we show that the number of iterations
required by EDDP for solving the SAA problem  is the same as that of DDP 
for solving the single-scenario problem. However,
to implement EDDP we need to maintain the set of saturated
search points explicitly.

Thirdly, we show that the SDDP method can be viewed as a randomized
version of the EDDP algorithm by choosing the aforementioned
feasible solution at each stage $t$ randomly from the $N_t$ possible 
selections. Since this algorithm is stochastic, we establish
the expected number of iterations required by SDDP
to compute an approximate feasible policy for solving the
SAA problem.
 In particular the iteration complexity of SDDP is worse than
 that of DDP and EDDP by a factor of $\bar N := N_2 \times N_3 \ldots N_{T-1}$,
 which increases exponentially w.r.t. $T$.
However, it may 
still have mild dependence on $T$ for the low accuracy region (see Section 5 for more discussions).
 Moreover, we show that the probability of
 having large deviation from this expected iteration
 complexity decays exponentially fast.
 In addition, we establish the convergence
 of the gap between a stochastic upper bound and 
 lower bound on the optimal value, and show how we can possibly use
 these bounds to terminate the algorithm.

To the best of our knowledge, all the aforementioned
complexity results, as well as the analysis techniques, are
new for cutting plane methods 
for multi-stage stochastic optimization.

This paper is organized as follows. In Section 2, we present some preliminary results
on the basic cutting plane methods for solving static convex optimization problems.
In Section 3, we present the DDP method for single-scenario problems and establish
its convergence properties. Section 4 is devoted to the EDDP method
for solving the SAA problem for multi-stage stochastic optimization.
In Section 5, we establish the complexity of the SDDP method.
Finally, some concluding remarks are made in Section~6.

\vgap

\section{Preliminary: Kelley's cutting plane methods}

In this section, we briefly review the basic cutting plane method and
establish its complexity bound.
Consider the convex programming problem of
\beq \label{cp}
\min_{x \in X} f(x),
\eeq
where $X \subseteq \bbr^n$ is a convex compact set and $f: X \to \bbr$ is a sub-differentiable convex function.
Moreover, we assume that $f$ is Lipschitz continuous s.t.
\beq \label{cp_lipschitz}
|f(x) - f(y)| \le M \|x - y\|, \forall x, y \in X.
\eeq

Algorithm~\ref{basic_cpm} formally
describes Kelley's cutting plane method for solving \eqnok{cp}.
The essential construct in this algorithm is the cutting plane model $\underline f(x)$,
which always underestimates $f(x)$ for any $x \in X$.
Given the current search point $x_k$, 
this method first updates the model function $\underline f$ and then minimizes
it to compute the new search point $x_{k+1}$.
It terminates if the
gap between the upper bound ($\ub_k$) and lower bound ($\lb_k$)  
falls within the prescribed target accuracy $\epsilon$. As a result, 
an $\epsilon$-solution $\bar x \in X$ s.t. $f(\bar x) - f(x^*) \le \epsilon$
will be found whenever the algorithm stops.

\begin{algorithm}[H]
\caption{Basic cutting plane method}
\begin{algorithmic}
\State {\bf Input:} initial points $x_1$ and target accuracy $\epsilon$.
\State Set $\underline f_0(x) = -\infty$ and $\ub_0 = + \infty$.
\For {$k =1,2,\ldots,$}
\State Set $\underline f_k(x) = \max\{ \underline f_{k-1}(x), f(x_k) + \langle f'(x_k), x - x_k \rangle\}$.
\State Set $x_{k+1} \in \Argmin_{x \in X} \underline f(x)$.
\State Set $\lb_k = \underline f(x_{k+1})$ and $\ub_k = \min\{ \ub_{k-1}, f(x_{k+1})\}$.
\If {$\ub_k - \lb_k \le \epsilon$}
\State {\bf terminate}.
\EndIf
\EndFor
\end{algorithmic} \label{basic_cpm}
\end{algorithm}

We establish the complexity, i.e., the number of iterations required to have a gap lower than $\epsilon$, 
of the cutting plane method in Proposition~\ref{prop:cuttingplane}.

\begin{proposition} \label{prop:cuttingplane}
Unless Algorithm~\ref{basic_cpm} stops, we have $\|x_{k+1} - x_i\| \ge \epsilon /M$ for any 
$i = 1, \ldots, k$. Moreover, suppose that the norm $\|\cdot\|$ in \eqnok{cp_lipschitz}
is given by the $l_\infty$ norm and $X \subset \bbr^n$ is contained
in a box with side length bounded by $l$.
Then the complexity of the basic cutting plane method can be bounded by
\beq \label{bnd_cuttingplane}
\left( \tfrac{l M}{\epsilon} +1\right)^n.
\eeq
\end{proposition}

\begin{proof}
Note that $\underline f_k(x) = \max_{i=1, \ldots,k} f(x_i) + \langle f'(x_i), x - x_i \rangle$
is Lipschitz continuous with constant $M$. Moreover, we have $\underline f_k(x) \le f(x)$ for any $x \in X$ and
$f(x_i) = \underline f_k(x_i)$ for any $i = 1, \ldots, k+1$. Hence,
\[
\underline f_k(x_{k+1}) = \min_{x \in X} \underline f_k(x) \le \min_{x \in X} f(x) = f^*.
\]
Using this observation, we have
\[
\ub_k - \lb_k \le f(x_i) - \lb_k  = \underline f_k(x_i) - \lb_k = \underline f_k(x_i) -\underline f_k(x_{k+1}) \le M \|x_i - x_{k+1}\|.
\]
Since $\ub_k - \lb_k > \epsilon$, we must have $\|x_i - x_{k+1}\| > \epsilon / M$.
\eqnok{bnd_cuttingplane} then follows immediately from this observation.
\end{proof}

\vgap

Even though the complexity bound~\eqnok{bnd_cuttingplane}
of the cutting plane method has not been explicitly 
established before, construction of this proof was used in Ruszczy\'{n}ski
\cite{Ruz03-1}. Moreover,
as pointed out in \cite{Nest04}  the exponential dependence of such complexity bound
on the dimension $n$ does not seem to be improvable in general.
It is worth noting that the cutting plane algorithm does not explicitly depend on the selection of the norm
even though the bound in \eqnok{bnd_cuttingplane} is obtained
under the assumption that $X$ sits inside an $l_\infty$ box.

\section{Dual dynamic programming for single-scenario problems}
In this section, we focus on a dynamic version
of the cutting plane method applied to solve
a class of deterministic dynamic convex optimization problems,
i.e., multi-stage optimization problems with a single scenario.
This dual dynamic programming (DDP) method, which can be viewed as
SDDP with one scenario, will serve as a starting point
for studying the more general dual dynamic programming methods
in later two sections. Moreover, this method may inspire some interests in its own right.

More specifically, 
we consider 
the following dynamic convex programming
\beq \label{dcp}
f^*:= \min_{x_1 \in X_1} \left\{ f_1(x_1) := h_1(x_1) + \lambda v_2(x_1)  \right\},
\eeq
where the value functions $v_t(\cdot)$, $t =2, \ldots, T+1$, are defined recursively by
\begin{align} 
v_t(x_{t-1}) &:= \min_{x_{t} \in X_t(x_{t-1})} \left\{f_t(x_t) := h_t(x_{t}) + \lambda v_{t+1}(x_{t})\right\}, \label{define_value_function}\\
v_{T+1}(x_T) &\equiv 0, \label{define_value_function1}
\end{align}
with convex feasible sets $X_t(x_{t-1})$ given by
\beq \label{dcp_feasiblesets}
X_t(x_{t-1}) := \left\{ x \in \bar X_t \subseteq \bbr^{n_t}: A_t x = B_t x_{t-1} + b_t, \phi_t(x) \le Q_t x_{t-1}\right\}.
\eeq
Similarly to problem~\eqnok{multstage}, here $\bar X_t \subset \bbr^{n_t}$ are closed convex  sets independent of $x_{t-1}$,
$\lambda \in (0,1]$ denotes
the discounting factor, $A_t: \bbr^{n_t} \to \bbr^{m_t}, B_t: \bbr^{n_{t-1}} \to \bbr^{m_t}$, and $Q_t: \bbr^{n_{t-1}} \to \bbr^{p_t}$ are linear mappings, 
 and $h_t: \bar X_t \to  \bbr$ and $\phi_{t,i}: \bar X_t \to  \bbr$, $i =1, \ldots, p_t$,
are closed convex functions.
Thus, we can view problem~\eqnok{dcp} as a single-scenario multi-stage optimization problem
in the form of~\eqnok{multstage}, by assuming $\xi_t =(\bmA_t,\bmb_t,\bmB_t, \bmQ_t, \bmp_t, \bmc_t)$
to be deterministic, and setting $h_t (\cdot) = H_t(\cdot, \bmc_t)$ and $\phi_t(\cdot) = \Phi_t(\cdot, \bmp_t)$. 

Throughout this section, we denote ${\cal X}_t$ the effective feasible region
of each period~$t$ defined recursively by
\beq \label{def_calX}
{\cal X}_t := 
\begin{cases}
X_1, & t=1,\\
 \cup_{x \in {\cal X}_{t-1}} X_t(x), &t \ge 2.
\end{cases}
\eeq 
Observe that ${\cal X}_t$ is not necessarily convex and its convex hull is 
denoted by ${\rm Conv}({\cal X}_t )$.
Moreover, letting ${\rm Aff}({\cal X}_t)$ be the affine hull of ${\cal X}_t$
and 
$
{\cal B}_t( \epsilon):= \{y \in {\rm Aff} ({\cal X}_t): \|y\| \le  \epsilon\}, 
$
we use
\[ 
{\cal X}_t(\epsilon) := {\cal X}_t + {\cal B}_t( \epsilon)
\]
to denote ${\cal X}_t$ together with its surrounding neighborhood.

In order to develop a cutting plane algorithm for solving problem \eqnok{dcp},
we need to make a few assumptions and discuss a few quantities that characterize the problem.

\begin{assumption} \label{assum2_bndX}
For any $t \ge 1$, there exists $D_t \ge 0$ s.t.
\beq \label{bound_X}
\| x_t - x'_t\| \le D_t, \ \ \forall x_t, x'_t \in {\cal X}_t, \ \forall t \ge 1. 
\eeq
\end{assumption}

The quantity $D_t$ provides a bound on the ``diameter" of the effective feasible region ${\cal X}_t$. 
Clearly, Assumption~\ref{assum2_bndX} holds if the convex sets $\bar X_t$
are compact, since by definition we have
$
{\cal X}_t \subseteq {\rm Conv}( {\cal X}_t) \subseteq \bar X_t, \ \forall t \ge 1.
$

\begin{assumption} \label{assum2_complete}
For any $t \ge 1$, there exists $\bar \epsilon_t \in (0, + \infty)$ s.t.
\begin{align}
h_t(x) < +\infty, \ \forall x \in {\cal X}_t(\bar \epsilon_t) \ \ \mbox{and} \ \ {\rm rint} (X_{t+1}(x)) \neq \emptyset, \ \forall x \in {\cal X}_t(\bar \epsilon_t), \label{suff_cond}
\end{align}
where ${\rm rint}(\cdot)$ denotes the relative interior of a convex set.
\end{assumption}
Assumption~\ref{assum2_complete} 
describes certain regularity conditions of problem~\eqnok{dcp}. Specifically, the two conditions
in \eqnok{suff_cond} imply that
$h_t$ and  $v_{t+1}$ are finitely valued in ${\cal X}_t(\epsilon_t)$.
The second relation in \eqnok{suff_cond} also implies the Slater condition of the feasible sets
in \eqnok{dcp_feasiblesets} and thus the existence of optimal dual solutions to define the cutting plane models
for problem~\eqnok{dcp}.  Here the relative interior is required due to the nonlinearity 
of the constraint functions in \eqnok{dcp_feasiblesets} and we can replace ${\rm rint} (X_{t+1}(x))$
with $X_{t+1}(x)$ if the latter is polyhedral.
Conditions of these types have been referred to as extended relatively
complete recourse, which is less stringent than imposing
complete recourse with $\bar \epsilon=+\infty$
in the second relation in \eqnok{suff_cond} (see~\cite{GirLecPhi15}).

\vgap

In view of Assumption~\ref{assum2_complete}, the objective functions $f_t$, as given by the summation of $h_t$ and $\lambda v_{t+1}$, must be 
finitely valued in ${\cal X}_t(\bar \epsilon_t)$. In addition, by Assumptions~\ref{assum2_bndX}  the set ${\cal X}_t$ 
is bounded. Hence the convex functions $f_t$ must be Lipschitz continuous over  ${\cal X}_t$ (see, e.g., Section~2.2.4 of \cite{LanBook2020}).
We explicitly state the Lipschitz constants of $f_t$ below since they will be used
in the convergence analysis our algorithm.

\begin{assumption}  \label{assum2}
For any $t \ge 1$, there exists $M_t \ge 0$ s.t.
\begin{align}
|f_t(x_t) - f_t(x'_t)| &\le M_t \| x_t - x'_t\|,  \ \ \forall x_t, x'_t \in {\cal X}_t \label{Lips_f}.
\end{align}
\end{assumption}

We are now ready to describe a dual dynamic programming method for solving problem~\eqnok{dcp} (see Algorithm~\ref{algo_dcpm}).
For notational convenience, we assume that $X_1(x_{0}^k) \equiv  X_1$ for any iteration $k \ge 1$.

\begin{algorithm}
\caption{Dual dynamic programming (DDP) for single-scenario problems}
\begin{algorithmic}[1]
\State Set $\underline v_t^0(x) = -\infty$, $t = 2, \ldots, T$, 
$\underline v_{T+1}^0 = 0$,
and $\ub_t^0 = + \infty$, $t=1, \ldots, T$.
\For {$k =1,2,\ldots,$}

\For {$t = 1, 2, \ldots, T$} \Comment{Forward phase.}
\beq \label{def_x_k_t}
x_t^k \in \Argmin \left\{ \underline f_t^{k-1}(x) := h_t(x) + \lambda \underline v_{t+1}^{k-1} (x): x \in X_t(x_{t-1}^k) \right\}.
\eeq
\EndFor



\State Set $\ub_1^k = \min\{\ub_1^{k-1}, \tsum_{t=1}^T \lambda^{t-1} h_t(x_t^k)\}$.

\State

\State Set $\underline v_{T+1}^k = 0$. \Comment{Backward phase.}
\For {$t = T, T-1, \ldots, 2$} 

\begin{align}
\tilde v^k_{t}(x_{t-1}^k) &= \min \left\{\underline f_t^k(x) := h_t(x) + \lambda \underline v_{t+1}^k(x): x \in X_t(x_{t-1}^k)\right\}. \label{def_lb_k}\\
 (\tilde v_t^k)'(x_{t-1}^k) &= [B_t, Q_t ] y_t^k, \mbox{where $y_t^k$ is the optimal dual multiplier of \eqnok{def_lb_k}.} \nn \\
 \underline v_{t}^k(x) &= \max \left \{\underline v_{t}^{k-1}(x),  
\tilde v_{t}^k(x_{t-1}^k) +  \langle (\tilde v_{t}^k)'(x_{t-1}^k), x - x_{t-1}^k \rangle  \right\}. \label{def_tilde_f}
\end{align}

\EndFor

\EndFor
\end{algorithmic}  \label{algo_dcpm}
\end{algorithm} 

We now make a few observations about the above DDP method.
Firstly, in the forward phase our goal is to compute a new policy $(x_1^k, x_2^k, \ldots, x_T^k)$ sequentially
starting from $x_1^k$ for the first stage. In this phase we utilize the cutting plane 
model $\underline v_{t+1}^{k-1}(\cdot)$ as a surrogate for the value function $v_{t+1}(\cdot)$
in order to approximate the objective function $f_t(\cdot)$ at stage $t$,
because we do not have a convenient expression for the value function $v_{t+1}(\cdot)$.
Since $(x_1^k, x_2^k, \ldots, x_T^k)$ is a feasible policy by definition, 
$ \tsum_{t=1}^T \lambda^{t-1} h_t(x_t^k)$ gives us an upper bound on the optimal value $f^*$ of problem \eqnok{dcp},
and accordingly, $\ub_1^k$ gives us the value associated with the best policy we found so far.

Secondly, given the new generated policy $(x_1^k, x_2^k, \ldots, x_T^k)$, our goal in the backward phase is to update the cutting plane models
$\underline v_t^{k-1}(\cdot)$ to $\underline v_t^{k}(\cdot)$, in order to provide
a possibly tighter approximation of $v_t(\cdot)$. 
More specifically,  
by Assumption~\ref{assum2_complete}, the feasible region of $X_t(x_{t-1}^k)$
of the subproblem in \eqnok{def_lb_k} has a nonempty relative interior. Hence 
the function
value $\tilde v_t^k(x_{t-1}^k)$ 
and the associated vector  $[B_t, Q_t ] y_t^k$ are well-defined, and
they define a supporting hyperplane for the approximate
value function $\tilde v_t^k(\cdot)$ defined in \eqnok{def_lb_k} (after replacing $x_{t-1}^k$ with any $x \in {\cal X}_{t-1}(\bar \epsilon_{t-1})$). 
Using all these supporting hyperplanes of $\tilde v_t^k$ that have been generated so far, we define
a cutting plane model $\underline v_t^k: \bbr^{n_t} \to \bbr$,  which underestimates
the original value function $v_t(\cdot)$ as shown in the following result.

\begin{lemma} \label{lem_bounding}
For any $k \ge 1$,
\begin{align}
 \underline v_t^{k-1}(x) \le  \underline v_t^k(x) \le \tilde v_t^k(x) \le  v_{t}(x), \forall x \in {\cal X}_{t-1}(\bar \epsilon_{t-1}), t=2, \ldots, T,  \label{cp_app2}\\
\underline f_{t}^{k-1}(x) \le  \underline f_{t}^{k}(x) \le f_t(x), \forall x \in {\cal X}_t(\bar \epsilon_t), t=1, \ldots, T. \label{cpp_app_bnd_tildef}
\end{align}
\end{lemma}

\begin{proof}
First observe that the inequalities in \eqnok{cpp_app_bnd_tildef} follow directly from \eqnok{cp_app2}
by using the facts that $f_t(x) = h_t(x) + \lambda v_{t+1}(x)$ and $\underline f_{t}^{k}(x) = h_t(x) + \lambda \underline v_{t+1}^k(x)$
due to the definitions of $f_t$ and $\underline f_t^k$ in \eqnok{define_value_function} and \eqnok{def_x_k_t}, respectively.
Moreover, the first relation $\underline v_t^{k-1}(x) \le  \underline v_t^k(x)$ follows directly from \eqnok{def_tilde_f}.

Second, we observe that the functions $\tilde v_t^k$ and  $v_{t}$ are well-defined over ${\cal X}_{t-1}(\bar \epsilon_{t-1})$ due to Assumption~\ref{assum2_complete}
and will show that the remaining inequalities in \eqnok{cp_app2}, i.e., $ \underline v_t^k(x) \le \tilde v_t^k(x) \le  v_{t}(x), \forall x \in {\cal X}_{t-1}(\bar \epsilon_{t-1})$,
 hold by using 
induction backwards for $t= T, \ldots, 1$ at any iteration $k$.
Let us first consider $t=T$.
Note that $\underline v_{T+1}^k = 0$
and thus by comparing the definitions of  $v_{T}(x)$ and $\tilde v_{T}^k(x)$
in \eqnok{define_value_function} and \eqnok{def_lb_k}, we have
$\tilde v_{T}^k(x) = v_{T}(x)$. 
Moreover, by definition
$  \tilde v_{T}^k(x_{T-1}^k) + \langle  (\tilde v_{T}^k)'(x_{T-1}^k), x - x_{T-1}^k \rangle$
is a supporting hyperplane of $\tilde v^k_{T}(x)$ at $x_{T-1}^k$. Combining these observations
with the definition of $\underline v_T^k(x)$ as a bundle of these supporting hyperplanes, 
we have
\begin{align}
\underline v_T^k(x) \le 
  \tilde v_{T}^k(x) = v_{T}(x). \label{last_stage_exact}
\end{align}
Now assume that $ \underline v_t^k(x) \le \tilde v_t^k(x) \le  v_{t}(x)$ for some $0 \le t \le T$.
Using the induction hypothesis of $ \underline v_t^k(x) \le  v_{t}(x)$ in the
the definitions of $v_{t-1}(x)$ and  $\tilde v^k_{t-1}(x)$ in \eqnok{define_value_function} and \eqnok{def_lb_k}, we conclude that
$\tilde v^k_{t-1}(x) \le   v_{t-1}(x)$.
Moreover, by definition $(\tilde v_{t-1}^k)'(x_{t-2}^k)$
is a subgradient of $\tilde v_{t-1}^k(\cdot)$ at $x_{t-2}^k$. Combining these relations,
we conclude
\beq \label{lb_k_t_rel}
 \tilde v_{t-1}^k(x_{t-2}^k)  + \langle (\tilde v_{t-1}^k)'(x_{t-2}^k), x - x_{t-2}^k \rangle
 \le  \tilde v_{t-1}^k(x) \le v_{t-1}(x),
\eeq
which clearly implies that $\underline v_{t-1}^k(x) \le \tilde v_{t-1}^k(x) \le v_{t-1}(x)$ by definition of $\underline v_{t-1}^k$.
\end{proof}

\vgap

In order to establish the complexity of Algorithm~\ref{algo_dcpm}, 
we need to show that the approximation functions $\underline f_t^k(\cdot)$ 
are Lipschitz continuous on ${\cal X}_t$.

\begin{lemma} \label{lemma_lip}
For any $t \ge 1$, there exists $\underline M_t \ge 0$ s.t.
\beq  \label{Lips_under_f}
|\underline f_t^k(x_t) - \underline f_t^k(x'_t) | \le \underline M_t  \| x_t - x'_t\|, \ \forall x_t, x'_t \in {\cal X}_t \ \forall \ k \ge 1. 
\eeq
\end{lemma}

\begin{proof}
 Note that by  Assumption~\ref{assum2_complete},
 for any $x \in {\cal X}_t(\bar \epsilon)$, the feasible region of $X_{t+1}(x)$ has a nonempty relative interior,
 hence for any $i = 1, \ldots,k$, the function
values $\tilde v_{t+1}^i(x_{t}^i)$ 
and the associated vectors  $[B_{t+1}, Q_{t+1} ] y_{t+1}^i$ are well-defined.
Therefore, the piecewise linear function $\underline v_{t+1}^k(x)$ given by
\[
\underline v_{t+1}^k(x) = \max_{i=1,\ldots, k} \tilde v_{t+1}^i(x_{t}^i) + \langle [B_{t+1}, Q_{t+1} ] y_{t+1}^i, x - x_{t}^i\rangle 
\]
is well-defined and sub-differentiable. This observation, in view of the convexity of $h_t$ and Assumption~\ref{assum2_complete},
then implies that 
$\underline f_t^k(x) = h_t(x) + \lambda \underline v_t^k(x)$ is  sub-differentiable on ${\cal X}_t$.
We now provide a bound for the subgradients $(f_t^k)'$ on ${\cal X}_t$.
Note that for any $x \in {\cal X}_t(\bar \epsilon)$ and $x_0 \in {\cal X}_t$, we have
\begin{align}
\langle (\underline f_t^k)'(x_0), x - x_0 \rangle &\le \underline f_t^k(x) - \underline f_t^k(x_0) \le f(x) -\underline f_t^1(x_0), \label{bnd_two_lips0}
\end{align}
where the last inequality follows from \eqnok{cpp_app_bnd_tildef}.
Letting $\|\cdot\|_* := \max_{\|x \| \le 1} \langle \cdot, x \rangle$ denotes
the conjugate norm of $\|\cdot\|$
and setting $x = x_0 + \bar \epsilon (\underline f_t^k)'(x_0) / \|(\underline f_t^k)'(x_0)\|_*$, we have 
\begin{align*}
\bar \epsilon \|(\underline f_t^k)'(x_0)\|_* 
\le  f(x) - \underline f_t^1(x_0)
\le \max_{x \in {\cal X}(\bar \epsilon)} f(x) - \min_{x \in {\cal X}} \underline f_t^1(x), 
\end{align*} 
which implies that
\[
\|(\underline f_t^k)'(x_0)\|_* \le \tfrac{1}{\bar \epsilon} [\max_{x \in {\cal X}_t(\bar \epsilon)} f(x) - \min_{x \in {\cal X}_t} \underline f_t^1(x)], \forall x_0 \in {\cal X}_t.
\]
The result in \eqnok{Lips_under_f} then follows directly from the above inequality, the boundedness of ${\cal X}_t$ and hence  ${\cal X}(\bar \epsilon)$,
and the fact that
\[
|\underline f_t^k(x_t) - \underline f_t^k(x'_t) | \le \max\{ \| (\underline f_t^k)'(x_t) \|_*, \| (\underline f_t^k)'(x'_t) \|_* \} \|x_t - x'_t\|
\]
due to the convexity of $\underline f_t$ and the Cauchy Schwarz inequality.
\end{proof}

\vgap

We now add some discussions about the Lipschitz continuity of $\underline f_t^k$ obtained in Lemma~\ref{lemma_lip}.
Firstly, it might be interesting to establish some relationship between the Lipschitz constants $\underline M_t$  and $M_t$ for
$\underline f_t^k$ and $f_t$, respectively. Under certain circumstances we can provide such a relationship.
In particular, let us suppose that 
\beq \label{close_f_t_under_f_t}
\underline f_t^k (x_0) \le f_t(x_0) \le \underline f_t^k(x_0) +\bar \epsilon.
\eeq
It then follows from the above assumption and \eqnok{bnd_two_lips0}
that 
\[
\langle (\underline f_t^k)'(x_0), x - x_0 \rangle \le f(x) - f(x_0) + \bar \epsilon.
\]
Setting $x = x_0 + \bar \epsilon \underline f'(x_0) / \|\underline f'(x_0)\|_*$, we conclude 
\begin{align*}
\bar \epsilon \|(\underline f_t^k)'(x_0)\|_* 
\le f(x) - f(x_0) + \bar \epsilon 
\le M \|x - x_0\| + \bar \epsilon
\le \bar \epsilon M + \bar \epsilon,
\end{align*} 
which implies that
\begin{align}
\|\underline f'(x_0)\|_* &\le M_t + 1 \ \ \mbox{and} \ \ \underline M_t \le M_t +1. \label{bnd_ftilde1} 
\end{align}
Note however that the above relationship does not necessarily hold for a situation more general than \eqnok{close_f_t_under_f_t}.

Secondly, while it is relatively easy to understand how the discounting factor $\lambda$ impacts the
Lipschitz constants $M_t$ for the objective functions $f_t$ over different stages,
its impact on the Lipschitz constants $\underline M_t$ for the approximation functions $\underline f_t^k$ is
more complicated since we do not know how the Lagrange multipliers $y_t^k$ changes w.r.t. $\lambda$.
On the other hand, the discounting factor does play a role in
compensating the approximation errors accumulated  over different stages for the DDP method.
Since we cannot quantify precisely such a compensation by simply scaling the Lipschitz constants $M_t$
and $\underline M_t$, we decide to incorporate explicitly the discounting factor
$\lambda$ into our problem formulation, as well as the analysis of our algorithms.
We will see that to incorporate $\lambda$ just makes some calculations, but not the
major development of the analysis, more complicated.
One can certainly assume that $\lambda = 1$ in order to see the basic idea of our convergence analysis.

\vgap

In order to establish the complexity of DDP, we need to 
introduce an important notion as follows.

\begin{definition} \label{def_sat}
We say that a search point $x_{t}^k$ gets $\epsilon_t$-{\sl saturated} 
at iteration $k$
if
\beq \label{def_saturation}
v_{t+1}(x_t^k) - \underline v_{t+1}^k(x_t^k) \le \epsilon_t. 
\eeq
\end{definition}

\vgap

In view of the above definition and \eqnok{cp_app2}, for any $\epsilon_t$-saturated 
point $x_t^k$ we must have
\begin{align}
 \underline v_{t+1}^k(x_t^k) \le v_{t+1}(x_t^k) \le \underline v_{t+1}^k(x_t^k) + \epsilon_t.
\end{align}
In other words,  
$\underline v_{t+1}^k$ will be a tight approximation of $v_{t+1}$ at
$x_t^k$ with error bounded by $\epsilon_t$.
 By \eqnok{cp_app2}, we also have
 $\underline v_{t+1}^k(x_t^k) \le \underline v_{t+1}^{k'}(x_t^k)$ for any $k' \ge k$,
 and hence
 \[
 v_{t+1}(x_t^k) - \underline v_{t+1}^{k'}(x_t^k) \le  v_{t+1}(x_t^k) - \underline v_{t+1}^{k}(x_t^k) \le
 \epsilon_t.
 \]
This implies that once
a point $x_t^k$ becomes $\epsilon_t$-saturated at the $k$-th iteration,
the functions $\underline v_{t+1}^{k'}$ will also be a tight approximation of $v_{t+1}$
at $x_t^k$ with error bounded by $\epsilon_t$ for any iteration $k' \ge k$.

\vgap

Below we describe some basic properties about the saturation of
the search points.

\begin{lemma} \label{lem:last_stage}
Any search point $x_{T-1}^k$
generated for the $(T-1)$-th stage must be $0$-saturated for any $k \ge 1$.
\end{lemma}

\begin{proof}
Note that by \eqnok{cp_app2}, we have
$\underline v_{T}^k(x_{T-1}^k) \le v(x_{T-1}^k)$.
Moreover,  
by \eqnok{def_tilde_f},
\[
\underline v_{T}^k(x_{T-1}^k) \ge \tilde v^k_{T} (x_{T-1}^k)  = v(x_{T-1}^k)
\]
where the last equality follows from the fact that $v_{T+1}^k = 0$ and the definitions
of $v_{T}(x)$ and $\tilde v_{T}^k(x)$
in \eqnok{define_value_function} and \eqnok{def_lb_k}.
Therefore we must have $\underline v_{T}^k(x_{T-1}^k) = v(x_{T-1}^k)$,
which, in view of \eqnok{def_saturation}, 
implies that $x_{T-1}^k$ is $0$-saturated.
\end{proof}

\vgap

We now state 
a crucial observation for DDP
that relates the saturation of search points
across two consecutive stages.
More specifically,
the following result shows that if one search point $x_t^j$ at stage $t$
has been $\epsilon_t$-saturated at iteration $j$, and a new search point generated
at a later iteration $k$ is close to $x_t^j$, then 
a search point in the previous stage $t-1$ will get $\epsilon_{t-1}$-saturated
with an appropriately chosen value for $\epsilon_{t-1}$.

\begin{proposition} \label{prop:saturation}
Suppose that the search point $x_t^k$ generated at 
the $k$-th iteration is close enough to $x_t^{j}$
generated in a previous iteration $ 1 \le j < k$, i.e., 
\beq \label{eq:closeenough}
\|x_t^k - x_t^j\| \le \delta_t 
\eeq
for some $\delta_t \in [0, +\infty)$.
Also assume that the search point $x_t^j$ is $\epsilon_t$-saturated, i.e.,
\beq \label{eq:existing_sat}
v_{t+1}(x^j_{t})  - \underline v_{t+1}^j(x^j_{t}) \le  \epsilon_{t}.
\eeq
Then we have
\begin{align} 
f_t(x_t^k) - \underline f_t^{k-1}(x_t^k) &= \lambda [v_{t+1}(x_t^k) - \underline v_{t+1}^{k-1}(x_t^k)] \nn\\
&\le \epsilon_{t-1}:= (M_t + \underline M_t) \delta_t + \lambda \epsilon_t. \label{gap_reducation}
\end{align}
In addition, for any $t \ge 2$, we have
\beq \label{saturation_cond}
v_{t}(x^k_{t-1})  - \underline v_{t}^k(x^k_{t-1}) \le  \epsilon_{t-1}
\eeq
and hence the
search point $x^k_{t-1}$ will get $\epsilon_{t-1}$-saturated at iteration $k$.
\end{proposition}

\begin{proof}
By the definitions of $f_t$ and $\underline f_t^{k-1}$ in \eqnok{define_value_function} and \eqnok{def_x_k_t} , we have
\[
f_t(x) -  \underline f_{t}^{k-1}(x) = \lambda[v_{t+1}(x)  - \underline v_{t+1}^{k-1}(x)], \forall x \in X_t(x_{t-1}^k)
\]
and hence first identity in \eqnok{gap_reducation} holds.
It follows from the definition of $x_t^k$ in \eqnok{def_x_k_t} and the first relation in \eqnok{cpp_app_bnd_tildef} that
\begin{align}
f_t(x_t^k) - \min_{x \in X_t(x_{t-1}^k)}  \underline f_t^{k-1} (x) &= f_t(x_t^k) - \underline f_t^{k-1}(x_t^k) \nn \\
&\le f_t(x_t^k) - \underline f_t^{j}(x_t^k). \label{temp_rel3}
\end{align}
Now by \eqnok{Lips_f} and \eqnok{Lips_under_f}, we have
\[
|f_t(x_t^k) - f_t(x_t^j)\| \le M_t \|x_t^k - x_t^j\| \ \ \mbox{and} \ \ |\underline f_t^{j}(x_t^k) - \underline f_t^{j}(x_t^j)| \le \underline M_t \|x_t^k - x_t^j\|.
\]
In addition, by \eqnok{eq:existing_sat} and 
the definition $f_t$ and $\underline f_t^{j}$, we have
\begin{align*}
f_t(x^j_{t}) -  \underline f_{t}^j(x^j_{t})
&= \lambda[v_{t+1}(x^j_{t})  - \underline v_{t+1}^j(x^j_{t})] \le  \lambda \epsilon_{t}.
\end{align*}
Combining the previous observations and \eqnok{eq:closeenough}, 
we have
\begin{align}
f_t(x_t^k) - \underline f_t^{k-1}(x_t^k)
&\le [f_t(x_t^k) - f_t(x_t^j)]
+ [f_t(x_t^j) - \underline f_t^{j}(x_t^j)] + [\underline f_t^{j}(x_t^j) - \underline f_t^{j}(x_t^k)] \nn \\
&\le (M_t + \underline M_t) \|x_t^k - x_t^j\| + \lambda \epsilon_t \nn \\
&\le (M_t + \underline M_t) \delta_t + \lambda \epsilon_t = \epsilon_{t-1},
 \label{temp_rel4}
\end{align}
where the last equality follows from the definition of $\epsilon_{t-1}$ in \eqnok{gap_reducation}.
Thus we have shown the inequality in \eqnok{gap_reducation}.

We will now show that the search point $x_k^{t-1}$ in the preceding stage $t-1$ must also be $\epsilon_{t-1}$-saturated at iteration $k$.
Note that $x_t^k$
is a feasible solution for the $t$-th stage problem and hence that the function value
$f_t(x_t^k)$ must be greater than the optimal value $v_{t}(x^k_{t-1})$.
Using this observation, we have
\begin{align}
v_{t}(x^k_{t-1}) - \underline v_{t}^k(x^k_{t-1})
\le f_t(x_t^k) - \underline v_{t}^k(x^k_{t-1}). \label{temp_rel1}
\end{align}
Moreover, using the definitions of  $\tilde v_t^k(x^k_{t-1}) $ and $\underline v_{t}^k(x^k_{t-1})$
in \eqnok{def_lb_k} and \eqnok{def_tilde_f}, the relations in \eqnok{cp_app2} and
 the fact that $\underline v_{t+1}^k(x) \ge \underline v_{t+1}^{k-1}(x)$
due to \eqnok{cpp_app_bnd_tildef}, we have
\begin{align}
\underline v_{t}^k(x^k_{t-1}) 
&= \max\{\underline v_{t}^{k-1}(x^k_{t-1}), \tilde v_t^k(x^k_{t-1}) \} \nn\\
&= \tilde v_t^k(x^k_{t-1}) \nn\\
&= \min \left\{ \underline f_t^k(x): x \in X_t(x_{t-1}^k)\right\} \nn \\
&\ge \min \left\{ \underline f_t^{k-1}(x): x \in X_t(x_{t-1}^k)\right\} \nn\\
&= \underline f_t^{k-1}(x_t^k), \label{temp_rel2}
\end{align}
where the last identity follows from the definition of $x_t^k$ in \eqnok{def_x_k_t}.
Putting together \eqnok{temp_rel1} and \eqnok{temp_rel2}, we have
\begin{align}
v_{t}(x^k_{t-1}) - \underline v_{t}^k(x^k_{t-1}) &\le  f_t(x_t^k) - \underline f_t^{k-1}(x_t^k) \nn \\
&\le \epsilon_{t-1}, \label{temp_rel3}
\end{align}
where the last inequality follows from \eqnok{temp_rel4}.
The above inequality then implies that $x^k_{t-1}$ gets $\epsilon_{t-1}$-saturated
at the $k$-th iteration.
\end{proof}

\vgap

Observe that the functions $f_t(\cdot)$ are not directly
computable since they depend on the exact value
functions $v_{t+1}(\cdot)$. The
following result relates
the notion of saturation to the gap between
a computable upper bound $\tsum_{t=1}^T  \lambda^{t-1} h_t(x_t^k)$
and the lower bound $\underline f_1^{k-1} (x_1^k)$
on the optimal value $f^*$, 
under the assumption that the concluding inequality \eqnok{gap_reducation} obtained in Proposition~\ref{prop:saturation}
holds for all the stages, i.e., $\lambda [v_{t+1}(x_t^k) -  \underline v_{t+1}^{k-1} (x_t^k)] \le \epsilon_{t-1}$, $\forall t=1, \ldots, T$.

\begin{lemma} \label{lem:bnd_gap}
Suppose that at some iteration $k \ge 1$, we have
\begin{align}
\lambda [v_{t+1}(x_t^k) -  \underline v_{t+1}^{k-1} (x_t^k)] &\le \epsilon_{t-1}, \label{gap_bnd_ass1}
\end{align}
for any $t =1, \ldots, T$. Then
we have
\beq \label{eq:validation}
\tsum_{t=1}^T \lambda^{t-1} h_t(x_t^k) -  \underline f_1^{k-1} (x_1^k) \le
 \tsum_{t=1}^T \lambda^{t-1} \epsilon_{t-1}.
\eeq
\end{lemma}

\begin{proof}
By 
the definition of $\underline f_1^{k-1} (x_1^k)$
in \eqnok{def_x_k_t}, we have
\[
\underline f_1^{k-1}(x_1^k) = h_1(x^k_1) + \lambda \underline v_2^{k-1}(x^k_1),
\]
which 
together with our assumption in \eqnok{gap_bnd_ass1} imply that
\begin{align} 
h_1(x^k_1) + \lambda v_2^{k-1}(x^k_1)   -  \underline f_1^{k-1} (x_1^k)  
= \lambda [ {v_2^{k-1}}(x^k_1)  -  \underline v_2^{k-1}(x^k_1)] 
\le \epsilon_0. \label{eq:hh_1}
\end{align}
Moreover, it follows from \eqnok{def_x_k_t} and \eqnok{cpp_app_bnd_tildef}  that 
\begin{align}
h_t(x^k_t) + \lambda \underline v_{t+1}^{k-1} (x^k_t) &= \min \left\{ \underline f_t^{k-1} (x): x \in X_t(x_{t-1}^k) \right\} \nn \\
&\le \min \left\{ f(x): x \in X_t(x_{t-1}^k) \right\}
= v_t(x_{t-1}^k), \nn
\end{align}
which, in view of our assumption
\[
\lambda [v_{t+1}(x^k_t) - \underline v_{t+1}^{k-1} (x^k_t)] \le  \epsilon_{t-1},
\]
then implies that
\beq \label{eq:hh_2}
h_t(x^k_t) + \lambda v_{t+1}(x^k_t)  \le v_t(x_{t-1}^k) + \epsilon_{t-1}
\eeq
for any $t =2, \ldots, T$.
Multiplying $\lambda^{t-1}$ to both side of the above inequalities, 
summing them up with the inequalities in \eqnok{eq:hh_1}, and using the fact that
$v_{T+1}(x_T^k) =0$, 
we have
\begin{align*}
\tsum_{t=1}^T \lambda^{t-1} h_t(x_t^k)   -  \underline f_1^{k-1} (x_1^k)
\le \tsum_{t=1}^{T} \lambda^{t-1} \epsilon_{t-1}.
\end{align*}
\end{proof}

\vgap

In the sequel, we use $S_t^{k-1}$ to denote the set of $\epsilon_t$-saturated search points
at stage~$t$ that have been generated 
 by the algorithm before the $k$-th iteration.
Using these sets, 
we now define the notion of {\sl distinguishable} search points as follows.

\begin{definition} \label{def_distinguishable}
We say that a search point $x_t^k$ at stage $t$
is $\delta_t$-distinguishable if 
\beq \label{eq:farenough}
g_{t}^k(x_t^k) > \delta_t,
\eeq
where $g_{t}^k(x)$ denotes the distance between $x$ to the set $S_t^{k-1}$
given by
\[
g_{t}^k(x) = 
\begin{cases}
\min_{s \in S_t^{k-1}} \|s - x\|, & t < T, \\
0, & \mbox{o.w.}
\end{cases}
\]
\end{definition}

Below we show that each iteration of the DDP method will either find an $\epsilon_0$-solution of problem~\eqnok{dcp},
or find a new $\epsilon_t$-saturated and $\delta_t$-distinguishable search point at some stage $t$ by properly specifying
$\delta_t$ and $\epsilon_t$ for $t = 0, \ldots, T-1$.

\begin{proposition} \label{prop_DDP_iteration}
Assume that $\delta_t \in [0, +\infty)$ for $t=1, \ldots, T$ are given. Also let us denote
\beq \label{eq:closeenough_dcp}
\epsilon_t := 
\begin{cases}
0, & t = T-1,\\
\tsum_{\tau=t}^{T-2} [ (M_{\tau+1} + \underline M_{\tau+1}) \delta_{\tau+1}\lambda^{\tau-t}], & t \le T-2.
\end{cases}
\eeq
Then, every iteration $k$ of the DDP method will
either generate a $\delta_t$-distinguishable and 
$\epsilon_t$-saturated search point $x_t^k$  at some stage $t = 1, \ldots, T$,
or find a feasible policy $(x_1^k, \ldots, x_T^k)$ of problem \eqnok{dcp} such that
\begin{align}
f_1(x_1^k) - f^* &\le \epsilon_0, \label{eq:bndOptGap} \\
\tsum_{t=1}^T  \lambda^{t-1} h_t(x_t^k) -  \underline f_1^{k-1} (x_1^k) &\le \tsum_{t=1}^{T} \lambda^{t-1} \epsilon_{t-1}. \label{eq:bndGap}
\end{align}
\end{proposition}

\begin{proof}
First note that the definition of $\epsilon_t$ is computed
according to the recursion $\epsilon_{t-1} = (M_t + \underline M_t) \delta_t + \lambda \epsilon_t$ (see \eqnok{gap_reducation})
and the assumption that $\epsilon_{T-1} = 0$.
Next, observe that exactly one of the following 
$T$ cases will happen at the $k$-th iteration of the DDP method.
\begin{itemize}
\item [] Case $1$: $g_{t}^k(x_{t}^k) \le \delta_t$, $\forall 1 \le t \le T-1$;
\item [] Case $t$, $t=2, \ldots, T-1$: $g_{i}^k(x_{i}^k) \le \delta_i$, $\forall t \le i \le T-1$, and $g_{t-1}^k(x_{t-1}^k) > \delta_{t-1}$;
\item [] Case $T$: $g_{T-1}^k(x_{T-1}^k) > \delta_{T-1}$.
\end{itemize}

We start with the first case. In this case, we have  $g_{t}^k(x_{t}^k) \le \delta_t$, $\forall 1 \le t \le T-1$.
Hence, $x_t^k$ must be close to an existing $\epsilon_t$-saturated
point $x_t^{j_t}$ for some $j_t \le k-1$ s.t.
\beq \label{close_enough_1}
\|x_t^k - x_t^{j_t}\| \le \delta_t, \ \forall  1 \le t \le T-1.
\eeq
It then follows from the above relation (with $t=1$), \eqnok{gap_reducation},
and the fact $f^* \ge \underline f_1^{k-1}(x_1^k)$ that
\begin{align}
f_1(x_1^k) - f^* &\le f_1(x_1^k) - \underline f_1^{k-1}(x_1^k) = \lambda [v_{2}(x_{1}^k) - \underline v_{2}^{k-1}(x_{1}^k)] \le \epsilon_0.
\end{align}
Moreover, we conclude from \eqnok{gap_reducation} and \eqnok{close_enough_1} that
\beq \label{gap_reduction_t0}
\lambda [v_{t+1}(x_{t}^k) - \underline v_{t+1}^{k-1}(x_{t}^k)] \le \epsilon_{t-1}, \forall 1 \le t \le T-1.
\eeq
Hence, the assumptions in Lemma~\ref{lem:bnd_gap}
hold and the result in \eqnok{eq:bndGap} immediately follows.

We now examine the $t$-th case for any  $2 \le t \le T-2$.
In these cases, we have  $g_{t-1}^k(x_{t-1}^k) > \delta_{t-1}$ and thus $x_{t-1}^k$ is $\delta_{t-1}$-distinguishable.
In addition, we have $g_{t}^k(x_{t}^k) \le \delta_{t}$. As a result, $x_t^k$ must be close to an existing $\epsilon_t$-saturated
point $x_t^{j_t}$ with $j_t \le k-1$.
This observation, in view of \eqnok{saturation_cond}, then implies that $v_{t}(x^k_{t-1})  - \underline v_{t}^k (x^k_{t-1}) \le   \epsilon_{t-1}$.
Hence $x^k_{t-1}$ is both $\delta_{t-1}$-distinguishable and $\epsilon_{t-1}$-saturated.

For the $T$-th case, we have $g_{T-1}^k(x_{T-1}^k) > \delta_{T-1}$
and hence  $x_{T-1}^k$ is $\delta_t$-distinguishable. Also by Lemma~\ref{lem:last_stage}, $x_{T-1}^k$ 
will get $0$-saturated. Therefore, $x_{T-1}^k$ is $\delta_{T-1}$-distinguishable and $\epsilon_{T-1}$-saturated (with $\epsilon_{T-1} = 0$).

Combining all these cases together, we conclude that 
 every DDP iteration will
either generate a $\delta_t$-distinguishable and 
$\epsilon_t$-saturated search point at some stage $t = 1, \ldots, T$,
or find a feasible policy of problem \eqnok{dcp}
satisfying \eqnok{eq:bndOptGap} and \eqnok{eq:bndGap}.
\end{proof}

\vgap

It is worth noting that each DDP iteration can possibly generate more than one
$\delta_t$-distinguishable and $\epsilon_t$-saturated 
points. For example, for the $t$-case in the 
the above proof of Proposition~\ref{prop_DDP_iteration}, we pointed out that
$x_{t-1}^k$ is $\delta_{t-1}$-distinguishable and $\epsilon_{t-1}$-saturated.
Some other search point $x_i^k$ with $i \le t-2$ in the preceding stages 
might also become $\delta_{i}$-distinguishable and $\epsilon_{i}$-saturated
even though there are no such guarantees.

\vgap

We are now ready to establish the complexity of the DDP method.
For the sake of simplicity, we will fix the norm $\|\cdot\|$  to be an $l_\infty$ norm
to define the distances and Lipschitz constants at each stage $t$.
It should be noted, however, that the DDP method itself
does not really depend on the selection of norms. The $l_\infty$ norm is chosen
because it will help us to count the number of search points needed
in each stage to guarantee the convergence of the algorithm.

\begin{theorem} \label{the_main}
Suppose that the norm used to define the bound on 
$D_t$ in \eqnok{bound_X} is the $l_\infty$ norm.
Also assume that $\delta_t \in [0, +\infty)$ are given and
that $\epsilon_t$ are defined in \eqnok{eq:closeenough_dcp}.
Then the number of iterations performed by the DDP method
to find a solution satisfying \eqnok{eq:bndOptGap} and \eqnok{eq:bndGap}
can be bounded by
\beq \label{def_N}
\tsum_{t=1}^{T-1} \left( \tfrac{D_t }{\delta_t} +1 \right)^{n_t} + 1.
\eeq
In particular, If $n_t \le n$, $D_t \le D$, $\max\{M_t, \underline M_t\} \le M$ and $\delta_t = \epsilon$ for all $t=1, \ldots, T$,
then the DDP method will find a feasible policy $(x_1^k, \ldots, x_T^k)$ of problem \eqnok{dcp} s.t.
\begin{align}
f_1(x_1^k) - f^* &\le 2 M   \min\{\tfrac{1}{1-\lambda}, T-1 \} \, \epsilon, \label{temination1} \\
\tsum_{t=1}^T  \lambda^{t-1} h_t(x_t^k) -  \underline f_1^{k-1} (x_1^k) &\le 2 M  \min\{ \tfrac{1}{(1-\lambda)^2}, \tfrac{T(T-1)}{2}\} \, \epsilon \label{temination2}
\end{align}
within at most
\beq \label{def_N_s}
(T-1)\left( \tfrac{D}{\epsilon} +1 \right)^{n} + 1
\eeq
iterations.
\end{theorem}

\begin{proof}
Let us count the total number of possible search points for
saturation before a solution satisfying \eqnok{eq:bndOptGap} and \eqnok{eq:bndGap}
is found. Using \eqnok{eq:farenough} and
the assumption the effective feasible region for each stage $t$
is inside a box with side length $D_t$ (c.f., \eqnok{bound_X}),
we can see that the number of possible $\delta_t$-distingushable search points for saturation
at each stage is given by
\[
N_t := \left( \tfrac{D_t }{\delta_t} +1 \right)^{n_t}.
\]
This observation together with Proposition~\ref{prop_DDP_iteration} 
then imply that the total number of
iterations performed by DDP will be bounded by $\tsum_{t=1}^{T-1} N_t + 1$
and hence by \eqnok{def_N}.

Now suppose that  $n_t \le n$, $D_t \le D$, $\max\{M_t, \underline M_t\}  \le M$ and $\delta_t = \epsilon$ for all $t=1, \ldots, T$.
We first provide a bound on $\epsilon_t$ defined in \eqnok{eq:closeenough_dcp}.  For $0 \le t \le T-2$, we have
\begin{align}
\epsilon_t &=
\tsum_{\tau=t}^{T-2} \lambda^{\tau-t}[(M_{\tau+1}+ \underline M_{\tau+1}) \delta_{\tau+1}] \nn \\
&=  2 M   \tsum_{\tau=t}^{T-2} \lambda^{\tau-t} \epsilon \nn\\
&\le 2 M   \min\{ \tfrac{1 - \lambda^{T-t-1}}{1-\lambda}, T-t-1 \} \epsilon \nn\\
&\le 2 M \min\{\tfrac{1}{1-\lambda}, T-t-1 \}  \epsilon, \label{def_epsilon_t_precisely}
\end{align}
and as a result, 
\begin{align}
\tsum_{t=1}^{T} \lambda^{t-1} \epsilon_{t-1}
&= \tsum_{t=1}^{T} \lambda^{t-1} \epsilon_{t-1} = \tsum_{t=0}^{T-2} \lambda^{t} \epsilon_{t} \nn\\
&\le 2 M  \tsum_{t=0}^{T-2} [\lambda^t \min\{\tfrac{1}{1-\lambda}, T-t-1 \}]  \epsilon \nn\\
&\le 2 M \tsum_{t=0}^{T-2} \min\{\tfrac{\lambda^t }{1-\lambda}, T-t-1 \} \epsilon \nn\\
&\le 2 M  \min\{ \tfrac{1}{(1-\lambda)^2}, \tfrac{T(T-1)}{2}\} \epsilon. \label{def_epsilon_t_precisely1}
\end{align}
Using these bounds in \eqnok{eq:bndOptGap} and \eqnok{eq:bndGap},
we obtain relations \eqnok{temination1} and \eqnok{temination2}.
Moreover, the iteration complexity bound in \eqnok{def_N_s} follows directly from \eqnok{def_N}.
\end{proof}

\vgap

We now add some remarks about the results obtained in Theorem~\ref{the_main}.

Firstly, similar to the basic cutting plane method, the bound in \eqnok{def_N} has an exponential dependence on
$n_t$. However, since the algorithm itself 
does not require us to explicitly discretize the decision variables in $\bbr^{n_t}$,
the complexity bound actually depends on the dimension of the affine space spanned by effective feasible region ${\cal X}_t$
defined in \eqnok{def_calX}, which can be smaller than the nominal dimension $n_t$.

Secondly, it is interesting to examine the dependence of the complexity bound in \eqnok{def_N_s} on
the number of stages $T$. In particular, if the discounting factor $\lambda < 1$,
the number of iterations required to find an $\epsilon$-solution of problem \eqnok{dcp}, i.e., a point $\bar x_1$ s.t. $f_1(\bar x_1) - f^* \le \epsilon$
only linearly depends on $T$. When the discounting factor $\lambda = 1$,
we can see that $T$ also appears in the termination criterions \eqnok{temination1}
and \eqnok{temination2}. 
As a result, the number of iterations required to find an $\epsilon$-solution of
problem\eqnok{dcp} will depend on $T^n$. The discounting factor provides a mechanism to compensate the errors
accumulated from approximating the value function $v_{t+1}$ by $\underline v_{t+1}^k$
starting from $t = T-1$ to $t =1$. 

Thirdly, while the termination criterion in \eqnok{temination1}
cannot be verified since the function value $f_1$ and $f^*$ are
not easily computable, the gap between the upper and lower bound in the l.h.s. of \eqnok{temination2} can be 
computed as we run the algorithm. It should be noted that
the dependence on $T$ for these two criterions are slightly different especially when
the discounting factor $\lambda =1$ (see the r.h.s. of \eqnok{temination1} and \eqnok{temination2}).

\section{Explorative dual dynamic programming} \label{sec_DDDP}
In this section, we generalize the DDP method for solving the multi-stage
stochastic optimization problems which have potentially an exponential number of scenarios.
As discussed in Section~\ref{sec_intro}, we assume that we can sample
from the probability distribution $P_t$ of the random vector $\xi_t$, $t = 2, \ldots, T$.
A sample average approximation (SAA) of the original problem \eqnok{multstage}
is constructed by replacing the true distribution of $\xi_t = (\bmA_t,\bmb_t,\bmB_t, \bmQ_t, \bmp_t, \bmc_t)$ with
the empirical distribution $P_{N_t}$ based on a random
sample
\[
\tilde \xi_{ti} = (\tilde A_{ti},\tilde b_{ti}, \tilde B_{ti}, \tilde Q_{ti}, \tilde p_{ti}, \tilde c_{ti}), i = 1, \ldots, N_t
\]
from the distribution $P_t$ of size $N_t$. Consequently the probability
distribution $P_2 \times \cdots \times P_T$ of the random process
$\xi_2, \ldots, \xi_T$ is replaced by $P_{N_2} \times \cdots \times P_{N_T}$.
Under the stage-wise independence assumption of $P_t$ and hence $P_{N_T}$,
it has been shown in \cite{Sha11} that under mild regularity assumptions we can
approximate problem~\eqnok{multstage1} by the SAA problem defined as

\begin{equation} \label{multstage1_SAA}
\begin{array}{ll}
F^* := \min_{x_1 \in X_1} \{ F_{11}(x_1):= H_1(x_1,c_1)+ \lambda  V_{2}(x_1)\},  
\end{array}
\end{equation}
where the value factions $V_t$, $t = 2, \ldots, T$, are recursively defined by
\begin{equation}\label{define_value_function_sddp}
\begin{array}{lll}
V_t(x_{t-1}) &:=& \tfrac{1}{N_t} \tsum_{i=1}^{N_t} \nu_{ti} (x_{t-1}),\\
\nu_{ti}(x_{t-1})  &:= & \min_{x_t \in X_t(x_{t-1}, \tilde \xi_{ti})} \{ F_{ti}(x_t):= H_t(x_t, \tilde c_{ti})+ \lambda V_{t+1}(x_t)\},
 \end{array}
 \end{equation}
 and
\begin{equation}\label{Defi_sto_V1_m_SAA}
 \begin{array}{lll}
  V_{T+1}(x_{T}) = 0.
\end{array}
\end{equation}

We will focus on how to solve the SAA problem in \eqnok{multstage1_SAA}.
The essential difference between this problem and the single-scenario problem in
\eqnok{dcp} is that each stage $t$ involves $N_t$ (rather than one) subproblems.
As a consequence, when determining the
search point $x_t^k$ at each stage
$t$ in the forward phase, we need to choose one out of $N_t$ feasible solutions and each one of
them corresponds to a 
realization $\tilde \xi_{ti}$ of the random variables. 
In this section, we will present a deterministic dual dynamic programming
method which chooses the feasible solution in the forward phase in an
aggressive manner, while in next section, we will discuss a stochastic approach
in which the feasible solution in the forward phase will be chosen randomly.
As we will see, the former approach will exhibit better
iteration complexity while the latter one is easier to implement.
We start with the deterministic approach also because
the analysis for the latter stochastic method is built on the one for the deterministic approach. 

Let ${\cal X}_{ti}$ be the effective feasible region for the $i$-th subproblem in stage $t$,
and $\bar {\cal X}_t$ be the effective feasible region all the subproblems in stage $t$, 
respectively, given by
\[
{\cal X}_{t i} :=
\begin{cases}
X_1, & t = 1,\\
\cup_{x \in \bar {\cal X}_{t-1}} X_t(x, \tilde \xi_{ti}),& t \ge 2,
\end{cases}
\]
and
\[
\bar {\cal X}_{t} :=
\begin{cases}
X_1, & t = 1,\\
\cup_{i=1,\ldots, N_t} {\cal X}_{ti}, &t \ge 2.
\end{cases}
\]
Observe that $\bar{\cal X}_t$ is not necessarily convex. 
Moreover, letting ${\rm Aff}(\bar{\cal X}_t)$ be the affine hull of $\bar{\cal X}_t$
and 
$
{\cal B}_{t}( \epsilon):= \{y \in {\rm Aff} (\bar{\cal X}_t): \|y\| \le  \epsilon\}, 
$
we use
\begin{align*}
\bar {\cal X}_t(\epsilon) &:= \bar {\cal X}_{t} + {\cal B}_{t}( \epsilon)
\end{align*}
to denote $\bar {\cal X}_t$ together with its small surrounding neighborhood.

We make the following assumptions throughout this section.

\begin{assumption} \label{assum2_bndX_m}
For any $t \ge 1$, there exists $D_t \ge 0$ s.t.
\beq \label{bound_X_sddp}
\| x_t - x'_t\| \le D_t, \ \ \forall x_t, x'_t \in \bar {\cal X}_t, \ \forall t \ge 1. 
\eeq
\end{assumption}

With a little abuse of notation, we still use $D_t$ as in the previous section
to bound the ``diameter" of the effective feasible region $\bar {\cal X}_t$. 
Clearly, Assumption~\ref{assum2_bndX_m} holds if the convex sets $\bar X_t$
are compact, since by definition we have
$
\bar {\cal X}_t \subseteq {\rm Conv}(\bar {\cal X}_t) \subseteq \bar X_t, \ \forall t \ge 1.
$

\begin{assumption} \label{assum2_complete_m}
For any $t \ge 1$, there exists $\bar \epsilon_t \in (0,+\infty)$ s.t.
\begin{align}
H_t(x, \tilde c_{ti})&< +\infty, \ \forall x \in \bar {\cal X}_{t}(\bar \epsilon_t), \forall i = 1, \ldots, N_t, \label{suff_cond_m1} \\
 {\rm rint} \left(X_{t+1}(x,\tilde \xi_{(t+1)i})\right) &\neq \emptyset, \ \forall x \in \bar {\cal X}_{t}(\bar \epsilon_t), \forall i = 1, \ldots, N_{t+1}, \label{suff_cond_m2}
\end{align}
where ${\rm rint}(\cdot)$ denotes the relative interior of a convex set.
\end{assumption}

Assumption~\ref{assum2_complete_m} 
describes certain regularity conditions of problem~\eqnok{multstage1_SAA}. Specifically, the conditions
in \eqnok{suff_cond_m1} and \eqnok{suff_cond_m2}
 imply that
$H_t(x, \tilde c_{ti})$ and  $V_{t+1}$ are finitely valued in a small neighborhood of ${\cal X}_{ti}$.
The second relation in \eqnok{suff_cond} also implies the Slater condition of the feasible sets
in \eqnok{define_value_function_sddp} and thus the existence of optimal dual solutions to define the cutting plane models
for problem~\eqnok{multstage1_SAA}.  Here the relative interior is required due to the nonlinearity 
of the constraint functions in \eqnok{define_value_function_sddp} and we can replace $ {\rm rint}\left (X_{t+1}(x,\tilde \xi_{(t+1)i})\right) $
with $X_{t+1}(x,\tilde \xi_{(t+1)i})$ if the latter is polyhedral.

\vgap

In view of Assumption~\ref{assum2_complete_m}, the objective functions $F_{ti}$
must be 
Lipschitz continuous over ${\cal X}_{ti}(\bar \epsilon_t)$. We explicitly state the Lipschitz constants of $F_{ti}$ below since they will be used
in the convergence analysis our algorithms. For the sake of notation convenience,
we still use $M_t$ to denote the Lipschitz constants for $F_{ti}$.

\begin{assumption}  \label{assum2_m}
For any $t \ge 1$ and $i = 1, \ldots, N_t$, there exists $M_t \ge 0$ s.t.
\begin{align}
|F_{ti}(x_t) - F_{ti}(x'_t)| &\le M_t \| x_t - x'_t\|, \ \ \forall x_t, x'_t \in {\cal X}_{ti} \label{Lips_f_sddp}.
\end{align}
\end{assumption}

\begin{algorithm}
\caption{Explorative dual dynamic programming (EDDP)}
\begin{algorithmic}[1] 

\State Set $\underline V_t^0(x) = -\infty$, $t = 2, \ldots, T$, $\underline V_{T+1}^{k} (x) = 0$, $k \ge 1$, and $S_t^0 = \emptyset$, $t=1, \ldots, T$.
\For {$k =1,2,\ldots,$}

\For {$t = 1, \ldots, T$}  \Comment{Forward phase.} 

\For {$i = 1, 2, \ldots, N_t$}

\begin{align}
\tilde x_{ti}^k &\in \Argmin_{x \in X_t(x_{t-1}^k, \tilde \xi_{ti})} \left\{ \underline F_{ti}^{k-1}(x) := H_t(x, \tilde c_{ti}) + \lambda \underline V_{t+1}^{k-1} (x)\right\}.  \label{def_x_k_t_sddp}\\
g_{t}^k(\tilde x_{ti}^k) &= 
\begin{cases}
\min_{s \in S_t^{k-1}} \|s - \tilde x_{ti}^k\|, & t < T, \\
0, & \mbox{o.w.}
\end{cases}
 \label{def_gap_k_t}
\end{align}

\EndFor

\State Choose $x_t^k$ from $\{\tilde x_{ti}^k\}$ such that
$
g_t^k(x_t^k) = \max\limits_{i=1, \ldots, N_t} g_{t}^k(\tilde x_{ti}^k).
$ \label{line_aggressive}
\EndFor

\vgap

\State {\bf if} {$g_1^k(x_1^k) \le \delta_0$} {\bf then}  Terminate.

\vgap

\For {$t = T, T-1, \ldots, 2$}    \Comment{Backward phase.}

\If {$g_t^k(x_t^k) \le \delta_t$}
\State Set $S_{t-1}^k = S_{t-1}^{k-1} \cup \{x_{t-1}^k\}$. \label{def_sat_set}
\EndIf

\For {$i = 1, \ldots, N_t$}

\begin{align}
\tilde \nu_{ti}^k(x_{t-1}^k) &= \min_{x \in X_t(x_{t-1}^k, \tilde \xi_{ti})} \left\{ 
\underline F_{ti}^{k}(x)  := H_t(x, \tilde c_{ti}) + \lambda \underline V_{t+1}^{k} (x) \right\}. \label{def_lb_k_sddp}\\
(\tilde \nu_{ti}^k)'(x_{t-1}^k) &= [\tilde B_{ti}, \tilde Q_{ti} ] y_{ti}^k, 
 \mbox{where $y_{ti}^k$ is the optimal} \nn\\
&\quad \quad \quad  \mbox{dual multipliers of \eqnok{def_lb_k_sddp}}.\label{def_lb_k_sddp1}
\end{align}

\EndFor

\State 
\begin{align}
& {\tilde V}_t^k = \tfrac{1}{N_t} \tsum_{j=1}^{N_t} \tilde \nu_{tj}^k(x_{t-1}^k),
({\tilde V_t^k})' =\tfrac{1}{N_t} \tsum_{j=1}^{N_t} (\tilde \nu_{tj}^k)'(x_{t-1}^k). \label{def_tilde_f_sddp0}\\
&\underline V_{t}^k(x) = \max \left \{\underline V_{t}^{k-1}(x),  {\tilde V}_t^k  + \langle ( {\tilde V_t^k})' ,
x - x_{t-1}^k \rangle \right\}. \label{def_tilde_f_sddp}
\end{align}
\EndFor

\EndFor
\end{algorithmic} \label{algo_sddp}
\end{algorithm}

\vgap

We now formally state the explorative dual dynamic programming (EDDP) method as shown in Algorithm~\ref{algo_sddp}.
A distinctive feature of EDDP is that it maintains a set of saturated search points $S_t^k$
for each stage $t$.
Similar to Definition~\ref{def_sat}, we 
say that a search point $x_{t}^k$ generated by the EDDP method is $\epsilon_t$-{\sl saturated} 
at iteration $k$
if
\beq \label{def_saturation_sddp}
V_{t+1}(x_t^k) - \underline V_{t+1}^k(x_t^k) \le \epsilon_t. 
\eeq
Moreover, similar to Definition~\ref{def_distinguishable},
we say an $\epsilon_t$-saturated search point $x_t^k$ at stage $t$
is $\delta_t$-distinguishable if 
\[
\| x_t^k - x_t^j\| > \delta_t
\]
for all other $\epsilon_t$-saturated search points $x_t^j$ that have been generated 
for stage $t$ so far by the algorithm. Equivalently, an $\epsilon_t$-saturated search point $x_t^k$
is $\delta_t$-distinguishable if
\beq \label{eq:farenough_sddp}
g_t^k(x_t^k) > \delta_t.
\eeq
Here $g_t^k(x_t^k)$ (c.f., \eqnok{def_gap_k_t}) denotes the distance between $x_{t}^k$ 
to the set $S_t^{k-1}$, i.e., the set of currently saturated search points in stage $t$.
Similar to the DDP method, 
saturation is defined for two given related sequences $\{\epsilon_t\}$ and $\{\delta_t\}$. More precisely,
the proposed algorithm takes $\{\delta_t\}$ as an initial argument and ends with $\{\epsilon_t\}$ (derived from $\{\delta_t\}$) 
saturated points.

In the forward phase of EDDP, for each stage $t$, we solve $N_t$ subproblems as shown in \eqnok{def_x_k_t_sddp}
to compute the search points $\tilde x_{ti}^k$, $i = 1, \ldots, N_t$.
For each  $\tilde x_{ti}^k$, we further compute the quantity $g_{t}^k(\tilde x_{ti}^k)$ in \eqnok{def_gap_k_t}, i.e.,
the distance between $\tilde x_{ti}^k$ and the set $S_t^{k-1}$ of currently saturated search points in stage $t$. 
Then we will choose
from $\tilde x_{ti}^k$, $i = 1, \ldots, N_t$, the one with the largest value of $g_{t}^k(\tilde x_{ti}^k)$
as $x_t^k$, i.e., 
$g_t^k(x_t^k) = \max_{i=1, \ldots, N_t} g_{t}^k(\tilde x_{ti}^k)$.
We can break the ties arbitrarily (or randomly to be consistent with the algorithm in the next section).
The search point $x^k_t$ is deemed to be saturated
if $g^k_t (x^k_t)$ is small enough, therefore so is the case for $\tilde x^k_{ti}$ for all $i$.
As a consequence, the point $x^k_{t-1}$ must also be saturated and can be added to $S^k_{t-1}$.
We call the sequence $(x_1^k, \ldots, x_T^k)$
a forward path at iteration $k$, since it is the trajectory generated in the forward phase
for one particular scenario of the data process $\tilde \xi_{ti}$.
In view of the above discussion, the EDDP method always chooses the most ``distinguishable" forward path to encourage exploration
in an aggressive manner (See Line~\ref{line_aggressive} of Algorithm~\ref{algo_sddp}). This also explains the origin
of the name EDDP.

The backward phase of EDDP is similar to the DDP in Algorithm~\ref{algo_dcpm} with
the following differences. First, we need to update the set $S_t^k$ for the saturated search points.
Second, the computation of the cutting plane model also requires the solutions of $N_t$ subproblems 
in \eqnok{def_lb_k_sddp}.

The following result is similar to Lemma~\ref{lem_bounding} for the DDP method.

\begin{lemma}
For any $k \ge 1$,
\begin{align}
 \underline V_t^{k-1}(x) \le  \underline V_t^k(x) \le \tfrac{1}{N_t}\tsum_{j=1}^{N_t} \tilde \nu_{tj}^k(x) \le  V_{t}(x), \forall x \in \bar {\cal X}_{t-1}(\bar \epsilon_{t-1}), t=2, \ldots, T,  \label{cp_app2_sddp}\\
\underline F_{ti}^{k-1}(x) \le  \underline F_{ti}^{k}(x) \le F_{ti}(x), \forall x \in \bar {\cal X}_{t}(\bar \epsilon_t), t=1, \ldots, T, i =1, \ldots, N_t. \label{cpp_app_bnd_tildef_sddp}
\end{align}
\end{lemma}

\begin{proof}
The proof is similar to that of Lemma~\ref{lem_bounding}.
The major difference exists in that \eqnok{lb_k_t_rel} will be replaced by
\begin{align*}
\underline V_{t-1}^k(x) &=\tfrac{1}{N_{t-1}} \tsum_{j=1}^{N_{t-1}}  
\left[\tilde \nu_{(t-1)j}^1(x_{t-2}^k)  + \langle (\tilde \nu_{(t-1)j}^k)'(x_{t-2}^k), x - x_{t-2}^k \rangle\right] \nn\\
&\le \tfrac{1}{N_{t-1}} \tsum_{j=1}^{N_{t-1}} \tilde \nu_{(t-1)j}^k(x) \le  \tfrac{1}{N_{t-1}} \tsum_{j=1}^{N_{t-1}}  \nu_{(t-1)j}^k(x) \nn\\
&= V_{t-1}(x),  
\end{align*}
and hence we skip the details.
\end{proof}

\vgap

In order to establish the complexity of the EDDP Algorithm, 
we need to show that the approximation functions $\underline F_{ti}^k(\cdot)$ 
are Lipschitz continuous on ${\cal X}_{ti}$. For convenience,
we still use $\underline M_t$ to denote the Lipschitz constants for $\underline F_{ti}^k$.
We skip its proof since it
is similar to that of Lemma~\ref{lemma_lip} 
after replacing Assumption~\ref{assum2_complete} with Assumption~\ref{assum2_complete_m}.

\begin{lemma} \label{lemma_lip_sddp}
For any $t \ge 1$ and $i=1, \ldots, N_t$, there exists $\underline M_t \ge 0$ s.t.
\beq  \label{Lips_underline_f_sddp}
|\underline F_{ti}^k(x_t) - \underline F_{ti}^k(x'_t)| \le \underline M_t \| x_t - x'_t\|, \ \ \forall x_t, x'_t \in \bar {\cal X}_{t}(\bar \epsilon_t) \ \forall \ k \ge 1. 
\eeq
\end{lemma}

\vgap

Below we describe some basic properties about the saturation of
search points.

\begin{lemma} \label{lem:last_stage_sddp}
Any search point $x_{T-1}^k$
generated for the $(T-1)$-th stage in EDDP must be $0$-saturated for any $k \ge 1$.
\end{lemma}

\begin{proof}
Note that by \eqnok{cp_app2_sddp}, we have
$\underline V_{T}^k(x_{T-1}^k) \le V(x_{T-1}^k)$.
Moreover, by \eqnok{def_tilde_f_sddp},
\begin{align*}
\underline V_{T}^k(x_{T-1}^k) &\ge \tfrac{1}{N_T} \sum_{j=1}^{N_t} \left[\tilde \nu_{tj}^k(x_{t-1}^k) + \langle (\tilde \nu_{tj}^k)'(x_{t-1}^k), \tilde \xi_{tj}),
x_{T-1}^k - x_{t-1}^k \rangle  \right] \\
&= \tfrac{1}{N_T} \tsum_{j=1}^{N_t} \tilde \nu_{tj}^k(x_{t-1}^k)  = \tfrac{1}{N_T} \tsum_{j=1}^{N_t} \nu_{tj}(x_{t-1}^k)   \\
&= V(x_{T-1}^k)
\end{align*}
where the second-to-last equality follows from the fact that $v_{T+1}^k = 0$ and the definitions
of $\nu_{Tj}(x)$ and $\tilde \nu_{Tj}^k(x)$
in \eqnok{define_value_function_sddp} and \eqnok{def_lb_k_sddp}.
Therefore we must have $\underline V_{T}^k(x_{T-1}^k) = V(x_{T-1}^k)$,
which, in view of \eqnok{def_saturation_sddp}, 
implies that $x_T^k$ is $0$-saturated.
\end{proof}

\vgap

We now generalize the result in Proposition~\ref{prop:saturation}
for the DDP method to relate the saturation of search points
across two consecutive stages in the EDDP method.

\begin{proposition} \label{prop:saturation_sddp}
Assume that $\delta_t \in [0, +\infty)$ for $t=1, \ldots, T$ are given and that 
$\epsilon_t$ are defined recursively according to \eqnok{gap_reducation}
for some given $\epsilon_{T-1} >0$. Also
let $g_t^k(\cdot)$ be defined in \eqnok{def_gap_k_t} and assume that $x_t^k$
is chosen such that
\[
g_t^k(x_t^k) = \max\limits_{i=1, \ldots, N_t} g_{t}^k(\tilde x_{ti}^k).
\]
\begin{itemize}
\item [a)] If $g_t^k(x_t^k) \le \delta_t$, $t =2, \ldots, T-1$, then
we have 
\beq \label{gap_reduction_sddp}
F_{ti}(\tilde x_{ti}^k) - \underline F_{ti}^{k-1}(\tilde x_{ti}^k)
= \lambda [V_{t+1}(\tilde x_{ti}^k)  - \underline V_{t+1}^{k-1}(\tilde x_{ti}^k)] 
\le \epsilon_{t-1}.
\eeq
Moreover, for any $T \ge 2$, we have
\beq \label{saturation_cond_sddp}
V_{t}(x^k_{t-1})  - \underline V_{t}^k(x^k_{t-1}) \le  \epsilon_{t-1}.
\eeq
where $\epsilon_{t-1}$ is defined \eqnok{gap_reducation}.

\item [b)] $S_t^k$, $t=1, \ldots, T-1$, contains all the $\epsilon_{t}$-saturated search points 
at stage $t$ generated by the algorithm up to the $k$-th iteration.
\end{itemize}
\end{proposition}

\begin{proof}
We prove the results by induction.
First note that by  \eqnok{def_gap_k_t} we have $g_T^k(x_T^k) =0$.
Moreover, by Lemma~\ref{lem:last_stage_sddp},
any search point $x_{T-1}^k$ will be $0$-saturated and hence part a) holds 
with $\epsilon_{T-1} =0$ for $t = T-1$.
Moreover, in view of Line~\ref{def_sat_set} of Algorithm~\ref{algo_sddp}
and the fact $g_T^k(x_T^k) =0$, $S_{T-1}^k$ contains
all the $0$-saturated search point obtained for stage $T-1$ and hence part b) holds
for $t = T-1$. 

Now assume that $g_t^k(x_t^k) \le \delta_t$ for the $t$-th stage for some $t \le T-1$.
In view of this assumption and the definition of $x_t^k$, 
we have 
\[
g_{t}^k(\tilde x_{ti}^k) = \min_{s \in S_t^{k-1}} \|s -  \tilde x_{ti}^k\| \le \delta_t
\]
for any $i = 1, \ldots, N_t$. Note that we must have $S_t^{k-1} \neq \emptyset$ since otherwise $g_{t}^k(\tilde x_{ti}^k) = + \infty$.
Hence, there exists  $x_t^{j_i} \in S_t^{k-1}$ for some $j_i < k-1$ such that
\begin{align}
\|x_t^{j_i} -  \tilde x_{ti}^k\| &\le \delta_t, \label{sddp_closeness}\\
V_{t+1}(x_t^{j_i}) -  \underline V_{t+1}^{j_i}(x_t^{j_i}) &\le \epsilon_t, \label{eq:existing_sat_sddp}
\end{align}
for any $t = 1, \ldots, N_t$.

Observe that by the definition fo $x_{ti}^k$ in \eqnok{def_x_k_t_sddp} and the first relation in \eqnok{cpp_app_bnd_tildef_sddp},
we have
\begin{align}
F_{ti}(\tilde x_{ti}^k) - \min_{x \in X_{ti}(x_{t-1}^k)}  \underline F_{ti}^{k-1} (x) &= F_{ti}(\tilde x_{ti}^k) - \underline F_{ti}^{k-1}(\tilde x_{ti}^k) \nn \\
&\le F_{ti}(\tilde x_{ti}^k) - \underline F_{ti}^{j_i}(\tilde x_{ti}^k). \label{temp_rel3}
\end{align}
Moreover, by \eqnok{Lips_f_sddp} and \eqnok{Lips_underline_f_sddp}, we have
\[
|F_{ti}(\tilde x_{ti}^k) - F_{ti}(x_t^{j_i}) | \le M_t \|\tilde x_{ti}^k - x_t^{j_i}\| \ \ \mbox{and} \ \
|\underline F_{ti}^{j_i}(\tilde x_{ti}^k) - \underline F_{ti}^{j_i}(x_t^{j_i})| \le \underline M_t \|\tilde x_{ti}^k - x_t^{j_i}\|.
\]
In addition, it follows from the definitions of $F_{ti}$ and $\underline F_{ti}^k$ 
(c.f. \eqnok{define_value_function_sddp} and \eqnok{def_x_k_t_sddp}) and \eqnok{eq:existing_sat_sddp} that
\begin{align*}
F_{ti}(x^{j_i}_{t}) -  \underline F_{ti}^{j_i}(x^{j_i}_{t})
&= \lambda[V_{t+1}(x^{j_i}_{t})  - \underline V_{t+1}^{j_i}(x^{j_i}_{t})] \le  \lambda \epsilon_{t}.
\end{align*}
Combining the previous observations and \eqnok{eq:existing_sat_sddp}, 
we have
\begin{align}
&F_{ti}(\tilde x_{ti}^k) - \underline F_{ti}^{k-1}(\tilde x_{ti}^k) \nn\\
&\le [F_{ti}(\tilde x_{ti}^k) - F_{ti}(x_t^{j_i})]
+ [F_{ti}(x_t^{j_i}) - \underline F_{ti}^{j_i}(x_t^{j_i})] + [\underline F_{ti}^{j_i}(x_t^{j_i}) - \underline F_{ti}^{j_i}(\tilde x_{ti}^k)] \nn \\
&\le (M_t+\underline M_t) \|\tilde x_{ti}^k - x_t^{j_i}\| + \lambda \epsilon_t \nn \\
&\le (M_t+\underline M_t) \delta_t + \lambda \epsilon_t = \epsilon_{t-1},
 \label{temp_rel4_sddp}
\end{align}
where the last inequality follows from
the definition of $\epsilon_{t-1}$ in \eqnok{gap_reducation}.
The above result, in view of the definitions of $F_{ti}$ and $\underline F_{ti}^k$,
then implies \eqnok{gap_reduction_sddp}.

We will now show that the search point $x_{t-1}^k$ in the preceding stage $t-1$ must also be $\epsilon_{t-1}$-saturated at iteration $k$.
Note that $\tilde x_{ti}^k$
are feasible solutions for the $t$-th stage problem and hence that the function value
$F_{ti}(\tilde x_{ti}^k)$ 
must be greater than the optimal value $\nu_{ti}(x^k_{t-1})$ defined in \eqnok{define_value_function_sddp}.
Using this observation, we have
\begin{align}
V_{t}(x^k_{t-1}) - \underline V_{t}^k(x^k_{t-1})
&= \tfrac{1}{N_t}\tsum_{i=1}^{N_t} \nu_{ti}(x^k_{t-1}) - \underline V_{t}^k(x^k_{t-1}) \nn\\
&\le \tfrac{1}{N_t}\tsum_{i=1}^{N_t} F_{ti}(\tilde x_{ti}^k) - \underline V_{t}^k(x^k_{t-1}). \label{temp_rel1_sddp}
\end{align}
Moreover, using the definitions of $\underline V_{t}^k(x^k_{t-1})$ and $\tilde \nu_{ti}^k(x^k_{t-1}) $
in \eqnok{def_tilde_f_sddp} and \eqnok{def_lb_k_sddp}, the relations in \eqnok{cp_app2_sddp} and
 the fact that $\underline V_{t+1}^k(x) \ge \underline V_{t+1}^{k-1}(x)$
due to \eqnok{cpp_app_bnd_tildef_sddp}, we have
\begin{align}
\underline V_{t}^k(x^k_{t-1}) 
&= \max\{\underline V_{t}^{k-1}(x^k_{t-1}), \tfrac{1}{N_t} \tsum_{i=1}^{N_t} \tilde \nu_{ti}^k(x_{t-1}^k) \} \nn\\
&= \tfrac{1}{N_t} \tsum_{i=1}^{N_t} \tilde \nu_{ti}^k(x_{t-1}^k) \nn\\
&= \tfrac{1}{N_t} \tsum_{i=1}^{N_t} \min \left\{ \underline F_{ti}^{k}(x): x \in X_t(x_{t-1}^k, \tilde \xi_{tj})\right\} \nn \\
&\ge \tfrac{1}{N_t} \tsum_{i=1}^{N_t} \min \left\{ \underline F_{ti}^{k-1}(x): x \in X_t(x_{t-1}^k, \tilde \xi_{tj})\right\} \nn\\
&= \tfrac{1}{N_t} \tsum_{i=1}^{N_t} \underline F_{ti}^{k-1}(\tilde x_{ti}^k), \label{temp_rel2_sddp}
\end{align}
where the last identity follows from the definition of $x_t^k$ in \eqnok{def_x_k_t_sddp}.
Putting together \eqnok{temp_rel1_sddp} and \eqnok{temp_rel2_sddp}, we have
\begin{align}
V_{t}(x^k_{t-1}) - \underline V_{t}^k(x^k_{t-1}) &\le   \tfrac{1}{N_t}\tsum_{i=1}^{N_t} [F_{ti}(\tilde x_{ti}^k) - \underline F_{ti}^{k-1}(\tilde x_{ti}^k)] \nn \\
&\le \epsilon_{t-1}, \label{temp_rel3_sddp}
\end{align}
where the last inequality follows from \eqnok{temp_rel4_sddp}.
The above inequality then implies that $x^k_{t-1}$ gets saturated
at the $k$-th iteration. Moreover, the point $x^k_{t-1}$ will be added
into the set $S_{t-1}^k$ in view of the definition in Line~\ref{def_sat_set} of Algorithm~\ref{algo_sddp}. 
We have thus shown both part a) and part b).
\end{proof}

\vgap

Different from the DDP method, we do not have a convenient way to compute 
an exact upper bound on the optimal value for the general multi-stage
stochastic optimization problem.
However, we can use $g_1^k(x_1^k)$
as a termination criterion for the EDDP method. Indeed, using  \eqnok{temp_rel4_sddp} (with $t=1$ and $i = 1$)
and the fact that $N_1 = 1$, we conclude that
if $g_1^k(x_1^k) \le \delta_1$, then we must have
\beq \label{EDDP_termination}
F_{11}(x_{1}^k) - F^* \le F_{11}(x_{1}^k) - \underline F_{11}^{k-1}(x_{1}^k) \le \epsilon_0.
\eeq
It is worth noting that one can possibly provide a stochastic upper bound
on $F^*$ for solving multi-stage stochastic optimization problems.
We will discuss this idea further in Section~\ref{sec_SDDP}.

\vgap

Below we show that each iteration of the EDDP method will either find an $\epsilon_0$-solution of problem~\eqnok{multstage1_SAA},
or find a new $\epsilon_t$-saturated and $\delta_t$-distinguishable search point at some stage $t$.

\begin{proposition}
Assume that $\delta_t \in [0, +\infty)$, $t=1, \ldots, T$, are given. Also let
$\epsilon_t$, $t=0, \ldots, T$, be defined in \eqnok{eq:closeenough_dcp}.
%
Then any iteration $k$ of the EDDP method will
either generate a new $\epsilon_t$-saturated and $\delta_t$-distinguishable search point $x_t^k$  at some stage $t = 1, \ldots, T$,
or find a feasible solution $x_1^k$ of problem \eqnok{multstage1_SAA} such that
\begin{align}
F_{11}(x_1^k) - F^* &\le \epsilon_0. \label{eq:bndOptGap_sddp} 
\end{align}
\end{proposition}

\begin{proof}
Similar to the proof of Proposition~\ref{prop_DDP_iteration}, we consider
the following $T$ cases that will happen at the $k$-th iteration of the EDDP method.
\begin{itemize}
\item [] Case $1$: $g_{t}^k(x_{t}^k) \le \delta_t$, $\forall 1 \le t \le T-1$;
\item [] Case $t$, $t=2, \ldots, T-1$: $g_{i}^k(x_{i}^k) \le \delta_i$, $\forall t \le i \le T-1$, and $g_{t-1}^k(x_{t-1}^k) > \delta_{t-1}$;
\item [] Case $T$: $g_{T-1}^k(x_{T-1}^k) > \delta_{T-1}$.
\end{itemize}
For the first case, it follows from  the assumption $g_{1}^k(x_{1}^k) \le \delta_1$ and \eqnok{EDDP_termination}  that
$x_1^k$ must be an $\epsilon_0$-solution of
problem~\eqnok{multstage1_SAA}. Now let us consider the $t$-th case for any $t = 2, \ldots, T-2$.
Since $g_{t-1}^k(x_{t-1}^k) > \delta_{t-1}$, the search point $x_{t-1}^k$ is $\delta_t$-distinguishable.
Moreover, we conclude from the assumption $g_{t}^k (x_{t}^k) \le \delta_{t}$ and
Proposition~\ref{prop:saturation_sddp}.a) that
the point $x_{t-1}^k$ must be $\epsilon_{t-1}$-saturated.
Hence, the search point $x_{t-1}^k$ is $\delta_t$-distinguishable and $\epsilon_{t-1}$-saturated
for the $t$-th case, $t = 2, \ldots, T-1$.
Finally for the $T$-th case, $x_{T-1}^k$ is $\delta_{T-1}$-distinguishable by assumption. Moreover,
by Lemma~\ref{lem:last_stage_sddp}, $x_{T-1}^k$ in the $(T-1)$-stage
will get $0$-saturated.  Hence $x_{T-1}^k$ is $\delta_{T-1}$-distinguishable and $\epsilon_{T-1}$-saturated.
The result then follows by putting all these cases together.
\end{proof}

\vgap

We are now ready to establish the complexity of the EDDP method.
For the sake of simplicity, we will fix the norm $\|\cdot\|$  to be an $l_\infty$ norm
to define the distances and Lipschitz constants at each stage $t$.

\begin{theorem} \label{the_main_EDDP}
Suppose that the norm used to define the bound on 
$D_t$ in \eqnok{bound_X_sddp} is the $l_\infty$ norm.
Also assume that $\delta_t \in [0,+\infty)$ are given and that
$\epsilon_t$ are defined in \eqnok{eq:closeenough_dcp}.
Then the number of iterations performed by the EDDP method
to find a solution satisfying 
\begin{align}
F_{11}(x_1^k) - F^* &\le \epsilon_0 \label{temination1_sddp} 
\end{align}
can be bounded by $\bar K + 1$, where
\beq \label{def_N_sddp}
\bar K:= \tsum_{t=1}^{T-1} \left( \tfrac{D_t }{\delta_t} +1 \right)^{n_t}.
\eeq
In particular, If $n_t \le n$, $D_t \le D$, $\max\{M_t,\underline M_t\} \le M$ and $\delta_t = \epsilon$ for all $t=1, \ldots, T$,
then the EDDP method will find a solution $x_1^k$ of problem \eqnok{multstage1_SAA} s.t.
\begin{align}
F_{11}(x_1^k) - F^* &\le 2 M  \min\{\tfrac{1}{1-\lambda}, T-1 \} \, \epsilon, \label{teminations_sddp} 
\end{align}
within at most $\bar K_\epsilon+1$ iterations with
\beq \label{def_N_s_sddp}
\bar K_\epsilon := (T-1)\left( \tfrac{D}{\epsilon} +1 \right)^{n}.
\eeq
\end{theorem}

\begin{proof}
Let us count the total number of possible search points for
saturation before an $\epsilon$-optimal policy of problem~\eqnok{multstage1_SAA}
is found. Using \eqnok{eq:farenough_sddp} and
the assumption the feasible region for each stage $t$
is inside a box with side length $D_t$ (c.f., \eqnok{bound_X_sddp}),
we can see that the number of possible search points for saturation
at each stage is given by
\[
\left( \tfrac{D_t}{\delta_t} +1 \right)^{n_t}.
\]
As a consequence, the total number of
iterations that EDDP will perform before finding
an $\epsilon_0$-optimal policy will be bounded
by $\bar K + 1$.
If $n_t \le n$, $D_t \le D$, $\max\{M_t,\underline M_t\} \le M$ and $\delta_t = \epsilon$ for all $t=1, \ldots, T$,
we can obtain \eqnok{teminations_sddp} 
by using the bound \eqnok{def_epsilon_t_precisely} for $\epsilon_0$ in \eqnok{temination1_sddp}.
Moreover, the bound in \eqnok{def_N_s_sddp} follows directly from \eqnok{def_N_sddp}.
\end{proof}

\vgap

We now add some remarks about the results obtained in Theorem~\ref{the_main_EDDP}
for the EDDP method. First, comparing with the DDP method for single-scenario problems,
we can see that these two algorithms exhibit similar iteration complexity. However, the DDP method 
provides some guarantees on an easily
computable gap between the upper and lower bound.
On the other hand, we can terminate the EDDP method by using the quantity $g_1^k$.
Second, the EDDP method requires us to maintain the set of saturated search points $S_t^k$
and explicitly use the selected norm $\|\cdot\|$ to compute $g_t^k$. 
In the next section, we will discuss
a stochastic dual dynamic programming method which can address some of these 
issues associated with EDDP, by sacrificing a bit on the iteration complexity bound
in terms of its dependence on the number of scenarios $N_t$.
Third, similar to the DDP method, we can
replace $n_t$ in the complexity bound of the EDDP method 
with the dimension of the effective region $\bar {\cal X}_t$ in \eqnok{def_N_sddp}.

\section{Stochastic dual dynamic programming} \label{sec_SDDP}

In this section, we still consider the SAA problem~\eqnok{multstage1_SAA}
for multi-stage stochastic optimization and suppose that Assumptions~\ref{assum2_bndX_m}, \ref{assum2_complete_m} and \ref{assum2_m} hold
 throughout this section.
 Our goal is to establish the iteration complexity of 
 the stochastic dual dynamic programming (SDDP) for solving
 this problem. 
 
As mentioned in the previous section, when dealing with
multiple scenarios in each stage $t$, we need 
to select $x_t^k$ 
from $\tilde x_{ti}$, $i = 1, \ldots, N_t$, defined in \eqnok{def_x_k_t_sddp},
where $\tilde x_{ti}$ corresponds to a particular realization $\tilde \xi_{ti}$, $i = 1, \ldots, N_t$.
While the EDDP method chooses $x_t^k$ in an aggressive manner by selecting the most ``distinguishable"
search points, SDDP will select
$x_t^k$ from $\tilde x_{ti}$, $i = 1, \ldots, N_t$, in a randomized
manner.

 
 The SDDP method is formally described in Algorithm~\ref{algo_sddp1}.
 This method still consists of the forward phase and backward phase
 similarly to the DDP and EDDP methods. 
 On one hand, we can view DDP as a special case of SDDP
 with $N_t = 1$, $t=1, \ldots, T$. On the other hand,
 there exist a few essential differences between SDDP in Algorithm~\ref{algo_sddp1}
and EDDP in Algorithm~\ref{algo_sddp}.
First, in the forward phase of SDDP, we randomly pick up an index $i_t$ and solve problem~\eqnok{def_x_k_t_sddp1} to update
$x_t^k$. Equivalently, one can view $x_t^k$ as being randomly chosen from $\tilde x_{t i}^k$, $i = 1, \ldots, N_t$,
defined in \eqnok{def_x_k_t_sddp} for the EDDP method.
Note that we do not need to compute $\tilde x_{t i}^k$ for $i \neq i_t$, even though they will be used in the analysis of
the SDDP method. Hence, the computation of the forward path $(x_1^k, \ldots, x_T^k)$ in SDDP
is less expensive than that in EDDP.
Second, in SDDP we do not need to maintain the set of saturated search points and thus
the algorithmic scheme is much simplified. However, without these sets, we will 
not be able to compute the quantities $g_t^k$ as in Algorithm~\ref{algo_sddp} and
thus cannot perform a rigorous termination test as in EDDP. We will 
discuss later in this section how to provide a statistical upper bound by running the forward phase  a few times.
 
\begin{algorithm}
\caption{Stochastic dual dynamic programming (SDDP)}
\begin{algorithmic}[1] 

\State Set $\underline V_t^0(x) = -\infty$, $t = 2, \ldots, T$, $\underline V_{T+1}^{k} (x) = 0$, $k \ge 1$.
\For {$k =1,2,\ldots,$}

\For {$t = 1, \ldots, T$}  \Comment{Forward phase.} 

\State Pick up $i_t \equiv i_t^k$ from $\{1, 2, \ldots, N_t\}$ uniformly randomly.

\State Set
\begin{align}
x_{t}^k &\in \Argmin_{x \in X_t(x_{t-1}^k, \tilde \xi_{t{i_t}})} \left\{ \underline F_{t{i_t}}^{k-1}(x) := H_t(x, \tilde c_{t{i_t}}) + \lambda \underline V_{t+1}^{k-1} (x)\right\}.  \label{def_x_k_t_sddp1}
\end{align}
\EndFor

\vgap
%

\For {$t = T, T-1, \ldots, 2$}    \Comment{Backward phase.}

\For {$i = 1, \ldots, N_t$}

\State Set $\tilde \nu_{ti}^k(x_{t-1}^k) $ and $(\tilde \nu_{ti}^k)'(x_{t-1}^k)$ according to \eqnok{def_lb_k_sddp} and \eqnok{def_lb_k_sddp1}.

\EndFor

\State Update $\underline V_{t}^k(x)$ according to \eqnok{def_tilde_f_sddp0} and \eqnok{def_tilde_f_sddp}.

\EndFor

\EndFor
\end{algorithmic} \label{algo_sddp1}
\end{algorithm}

As mentioned earlier, our goal in this section is to solve the SAA problem in \eqnok{multstage1_SAA} instead of the 
original problem in~\eqnok{multstage}.
Hence the randomness for the SDDP method in Algorithm~\ref{algo_sddp1}
comes from the i.i.d. random selection variable $i_t^k$ only. 
The statistical analysis to relate the SAA problem in  \eqnok{multstage1_SAA} and the original problem in \eqnok{multstage}
has been extensively studied especially under the stage-wise independence assumption (e.g. \cite{Sha11}).
The separation of these two problems allows us to greatly simplify the analysis of SDDP.

Whenever the iteration index $k$ is clear from the context, we
use the short-hand  notation $i_t \equiv i_t^k$.
We also use the notation
\[
i_{[k,t]} := \{i_1^1, \ldots, i_T^1, i_1^2,\ldots, i_T^2,\ldots, \ldots, i_1^{k-1}, \ldots, i_T^{k-1}, i_1^k, \ldots, i_t^k\}
\]
to denote the sequence of random selection variables generated up to stage $t$ at the $k$-th iteration.
The notions $i_{[k,0]}$ and $i_{[k-1,T]}$ will be used interchangeably.
We use ${\cal I}_{k,t}$ to denote the sigma-algebra generated by
$i_{[k,t]}$.
It should be noted that fo any iteration $k \ge 1$, we must have $i_1^k = 1$ since the number of
scenarios $N_1 = 1$. In other words, $i_1^k$ is alway deterministic for any $k \ge 1$.

The complexity analysis
of SDDP still relies on the concept of saturation.
Let us denote
$S_t^{k-1}$ the set of saturated points in stage $t$, i.e.,
$S_t^{k-1}:=\{x_t: V_{t+1}(x_t) -  \underline V_{t+1}^{j}(x_t) \le \epsilon_t, \mbox{ for some } j \le k-1 \}$.
We still use 
$x_t^{j_i}$ for some $j_i < k-1$
to denote the closest point to $\tilde x_{t i}^k$ from the saturated points $S_t^{k-1}$, i.e.,
\begin{align}
x_t^{j_i} \in \Argmin_{s \in S_t^{k-1}} \|s - \tilde x_{ti}^k\|,\\
V_{t+1}(x_t^{j_i}) -  \underline V_{t+1}^{j_i}(x_t^{j_i}) \le \epsilon_t. \label{eq:existing_sat_sddp1}
\end{align}
In SDDP, we will explore the average distance between $\tilde x_{ti}^k$ to the set $S_t^{k-1}$
defined as follows:
\begin{align}
\tilde g_t^k &:=  \tfrac{1}{N_t}\tsum_{i=1}^{N_t} \|\tilde x_{ti}^k - x_t^{j_i}\|. \label{def_tilde_g_t}
\end{align}
Note that the search point $x_t^k$ is a function of $i_{[k,t]}$ and hence is also random.
$\tilde x_{ti}^k$ depends on $x_{t-1}^k$ (see \eqnok{def_x_k_t_sddp}) and hence on $i_{[k, t-1]}$. 
Moreover, the set of saturated points $S_t^{k-1}$ only depends on $i_{[k-1,T]}$
since it is defined in the backward phase of the previous iteration.
Hence, $\tilde g_t^k$ is measurable w.r.t. ${\cal I}_{k,t-1}$, but it is independent of the random selection variable $i_t^k$
for the current stage $t$ at the $k$-th iteration.

Lemma~\ref{avg_distance} below summarizes some important properties about $\tilde g_t^k$.

\begin{lemma} \label{avg_distance}
Let $\delta_t \in [0, +\infty)$ be given and $\epsilon_t$ be defined in \eqnok{gap_reducation}.
If $\tilde g_t^k \le \delta_t$, then we have
\beq \label{recursion_sddp1}
 \tfrac{1}{N_t}\tsum_{i=1}^{N_t} [F_{ti}(\tilde x_{ti}^k) - \underline F_{ti}^{k-1}(\tilde x_{ti}^k)]
 =  \tfrac{\lambda}{N_t}\tsum_{i=1}^{N_t} [V_{t+1}(\tilde x_{ti}^k) - \underline V_{t+1}^{k-1}(\tilde x_{ti}^k)] \le \epsilon_{t-1}.
\eeq
Moreover, for $t \ge 2$ we have
\beq \label{saturation_sddp1}
V_{t}(x^k_{t-1}) - \underline V_{t}^k(x^k_{t-1}) \le \epsilon_{t-1}.
\eeq
\end{lemma}

\begin{proof}
First note the second inequality in \eqnok{temp_rel4_sddp} still holds since it does not
depend on the selection of $x_t^k$. Hence we have
\[
F_{ti}(\tilde x_{ti}^k) - \underline F_{ti}^{k-1}(\tilde x_{ti}^k) \le (M_t+\underline M_t) \|\tilde x_{ti}^k - x_t^{j_i}\| + \lambda \epsilon_t.
\]
Summing up the above inequalities, we can see that
\begin{align*}
 \tfrac{1}{N_t}\tsum_{i=1}^{N_t} [F_{ti}(\tilde x_{ti}^k) - \underline F_{ti}^{k-1}(\tilde x_{ti}^k)]
&\le (M_t+\underline M_t)  \tfrac{1}{N_t}\tsum_{i=1}^{N_t} \|\tilde x_{ti}^k - x_t^{j_i}\| + \lambda \epsilon_t  \\
&= (M_t+\underline M_t)  \tilde g_t^k + \lambda \epsilon_t\\
&\le \epsilon_{t-1},
\end{align*}
which together with the definitions of $F_{ti}$ and $\underline F_{ti}^{k-1}$ then imply \eqnok{recursion_sddp1}.
Moreover, \eqnok{saturation_sddp1} follows from \eqnok{temp_rel3_sddp} and \eqnok{recursion_sddp1}.
\end{proof}

\vgap

Similar to the previous section, we use
\begin{align*}
g_t^k(x_t^k) &:= \begin{cases}
\min_{s \in S_t^{k-1}} \|s- x_t^k\|, & t < T,\\
0, & \mbox{o.w.}
\end{cases}
\end{align*}
to measure the distance between $x_t^k$ and the set of saturated points.
Clearly, $g_t^k(x_t^k)$ is a random variable dependent on $x_t^k$ and
hence measurable w.r.t. ${\cal I}_{k,t}$.
We say that $x_t^k$ is $\epsilon_t$-saturated 
if $V_{t+1}^k(x_t^k) - \underline V_{t+1}(x_t^k) \le \epsilon_t$.
Moreover, $x_t^k$ is said to be $\delta_t$-distinguishable if $g_t^k(x_t^k) > \delta_t$.

The quantitates $\tilde g_t^k$ and $\tilde g_{t+1}^k$ defined in \eqnok{def_tilde_g_t}
provide us a way to check whether $x_t^k$ is $\delta_t$-distinguishable 
and $\epsilon_t$-saturated.
More specifically,
If $\tilde g_t^k > \delta_t$ for some stage $t <T$ at iteration $k$,
then there must exist an index $i_t^* \equiv i_t^{k,*} \in \{1, \ldots, N_t\}$
s.t. $\|\tilde x_{ti_t^*}^k - x_t^{j_{i_t^*}}\| > \delta$
or equivalently $g_t^k(\tilde x_{ti_t^*}^k) > \delta_t$
 (since otherwise $\tilde g_t^k \le \delta_t$).
Note that both $\tilde g_t^k$ and $i_t^*$ are measurable w.r.t. ${\cal I}_{k,t-1}$
but independent of
the $i_t^k$. Therefore, conditioning on ${\cal I}_{k,t-1}$
the probability of 
having $i_t^k = i_t^*$ is $1/N_t$, and
 consequently by the law of total probability,
$
\prob \{x^k_{t} =  \tilde x_{t i_t^*}^k \} = 1/N_{t}.
$
Moreover, we can see that
the conditional probability of 
\begin{align}
\prob\{ g_t^k(x_t^k) > \delta_t | \tilde g_t^k > \delta_t\}
&= \tsum_{i=1}^{N_t} \tfrac{1}{N_t} \prob\{g_t^k(\tilde x_{ti}^k) > \delta_t | \tilde g_t^k > \delta_t \} \nn\\
&\ge  \tfrac{1}{N_t} \prob\{ g_t^k(\tilde x_{ti_t^*}^k) > \delta_t  |\tilde g_t^k > \delta_t\} \nn \\
&= \tfrac{1}{N_t}. \label{bnd_dist_prob}
\end{align}
In other words, if $\tilde g_t^k > \delta_t$, then with probability at least $1/N_t$, $x_t^k$ will
be $\delta_t$-distinguishable.
If, in addition, $\tilde g_{t+1}^k \le \delta_{t+1}$,
then in view of Lemma~\ref{avg_distance}, we have
$V_{t+1}^k(x_t^k) - \underline V_{t+1}(x_t^k) \le \epsilon_t$
and hence $x_t^k$ will be $\epsilon_t$-saturated.

While EDDP can find at least one new saturated and distinguishable search point in every iteration,
SDDP can only guarantee so in probability as shown in the following result.
We use the random variable $\pIter^k$ to denote whether there exists such a point among any 
stages at iteration $k$. Clearly, 
$\pIter^k$ is measurable w.r.t. ${\cal I}_{k,T}$. 

\begin{lemma} \label{lemma_prob_saturation}
Assume that $\delta_t \in [0, +\infty)$, $t=1, \ldots, T$, are given. Also let 
$\epsilon_t$, $t=0, \ldots, T$, be defined in \eqnok{eq:closeenough_dcp}.
The probability of 
finding a new $\delta_t$-distinguishable and $\epsilon_t$-saturated and  search point at
the $k$-iteration of SDDP can be bounded by
\beq \label{bnd_iter_prob_0}
\prob\{\pIter^k = 1\}  \ge \tfrac{1}{\bar N} (1- \prob\{\tilde g_{i}^k \le \delta_{i}, i=1, \ldots, T-1\}),
\eeq
where 
\beq \label{def_max_N_t}
\bar N := \textstyle \prod_{i=2}^{T-1} N_i.
\eeq
\end{lemma}

\begin{proof}
Let $A$ denote that the event that that $\tilde g_{t}^k > \delta_{t}$ for some $t = 1, \ldots, T-1$.
Clearly we have $\prob\{A\} = 1- \prob\{\tilde g_{i}^k \le \delta_{i}, i=1, \ldots, T-1\}$.
Assume that the event $A$ happens.
Let $S$ denote the set of sample paths, i.e., selection of $T$ i.i.d. uniformly sample indices,
where there exists at least one index with  $\tilde g_{t}^k > \delta_{t}$.
Clearly we have $|S| \le \textstyle \prod_{t=2}^{T-1} N_t$, and each sample path occurs with
equal probability. We will show that there exists at least one sample path in $S$
that generates and selects an $\epsilon_t$-saturated and $\delta_t$-distinguishable search
point. Let us consider the following cases.
\begin{itemize}
\item [a)] There exists a sample path in $S$ such that $\tilde g_{T-1}^k > \delta_{T-1}$.
In this case, there exists at least one search point  $x_{T-1,i}^k$ such that
$g_{T-1}^k(x_{T-1,i}^k) > \delta_{T-1}$, since every search point in stage $T-1$
is $\epsilon_{T-1}$-saturated, we are done.
\item [b)] Amongst all sample paths, no path will have $\tilde g_{T-1}^k > \delta_{T-1}$.
Consider the set of sample paths with a stage $t$ such that $\tilde g_t^k > \delta_t$.
There exists at least one search point $x_{ti}^k$ such that $g_t^k(x_{ti}^k) > \delta_t$.
At least $1/N_t$ fraction of these sample paths will select $x_{ti}^k$ as the search
point. Now, one of the following two cases must occur upon selecting $x_t^k \leftarrow x_{ti}^k$:
\begin{itemize}
\item [b1)]The sample path will have $\tilde g_{t+1}^k \le \delta_{t+1}$.
Then, by Lemma~\ref{avg_distance}, $x_t^k$ will be $\epsilon_t$-saturaged.
Since we have already shown $x_t^k$ is also $\delta_t$-distinguishable, we are done.
\item [b2)] The sample path will have $\tilde g_{t+1}^k > \delta_{t+1}$.
Repeat the same argument with $t=t+1$. By the assumption, this incremental argument
must terminate since we cannot have a sample path with $\tilde g_{T-1}^k > \delta_{T-1}$.
\end{itemize}
\end{itemize}
In both cases, we have shown the existence of a sample path that generates and
selects an $\epsilon_t$-saturated and $\delta_t$-distinguishable search point.
Therefore, we have
\[
\prob\{ \pIter^k = 1| A\} \ge \tfrac{1}{|S|} \ge  \textstyle \prod_{t=2}^{N-1} \tfrac{1}{N_t}  = \tfrac{1}{\bar N},
\]
from which the result immediately follows.
\end{proof}
\vgap

In view of Lemma~\ref{lemma_prob_saturation}, one of the following three different cases
will happen for each SDDP iteration:
(a) $\tilde g_t^k \le \delta_t$ for all $t = 1, \ldots, T-1$.
The probability of this case is denoted by $\prob\{\tilde g_{i}^k \le \delta_{i}, i=1, \ldots, T-1\}$;
(b) A new $\epsilon_t$-saturated and $\delta_t$-distinguishable search
point will be generated with probability at least
\[
\tfrac{1}{\bar N} (1- \prob\{\tilde g_{i}^k \le \delta_{i}, i=1, \ldots, T-1\});
\]
and (c) none of the above situation will happen, implying that this particular
SDDP iteration is not productive.

%
%



\vgap

Observe that if for some iteration $k$, we have $\tilde g_t^k \le \delta_t$ for all $t = 1, \ldots, T-1$. 
Then by Lemma~\ref{avg_distance} (with $t=1$), we have
\beq \label{gap_bnd_ass1_sddp0}
F_{11}(x_1^k) - F^* \le F_{11}(x_1^k) - \underline F^{k-1}_{11}(x_1^k) \le \epsilon_0.
\eeq
Moreover, we have
\[
 \tfrac{\lambda}{N_t}\tsum_{i=1}^{N_t} [V_{t+1}(\tilde x_{ti}^k) - \underline V_{t+1}^{k-1}(\tilde x_{ti}^k)] \le \epsilon_{t-1}
\]
for all $t =1, \ldots, T$. This observation together with the fact that $x_t^k$ is randomly chosen
from $\tilde x_{ti}^k$, $i = 1, \ldots, N_t$, then imply that
the expectation of $V_{t+1}(x_t^k) -  \underline V_{t+1}^{k-1} (x_t^k)$ conditionally on $i_{[k,t-1]}$:
\begin{align}
\bbe[V_{t+1}(x_t^k) -  \underline V_{t+1}^{k-1} (x_t^k)| {\cal I}_{k,t-1}] 
&= \tfrac{\lambda}{N_t}\tsum_{i=1}^{N_t} [V_{t+1}(\tilde x_{ti}^k) - \underline V_{t+1}^{k-1}(\tilde x_{ti}^k)]  \nn\\
&\le \epsilon_{t-1}, t = 1, \ldots, T. \label{gap_bnd_ass1_sddp}
\end{align}
Similar in spirit to Lemma~\ref{lem:bnd_gap}, the following result
relates the above notion of saturation to the gap between a stochastic upper bound
and lower bound on the optimal value of problem~\eqnok{multstage1_SAA}.

\begin{lemma} \label{lem:bnd_gap_sddp}
Suppose that the relations in \eqnok{gap_bnd_ass1_sddp} hold for some iteration $k \ge 1$.
Then we have
\beq \label{eq:validation}
\tsum_{t=1}^T \lambda^{t-1} \bbe[H_t(x_t^k, \tilde c_{t i_t}) | {\cal I}_{k,t-1} ]  -  \bbe[ \underline F_{11}^{k-1} (x_1^k)| {\cal I}_{k-1,T}]\le
 \tsum_{t=1}^T \lambda^{t-1} \epsilon_{t-1}.
\eeq
\end{lemma}

\begin{proof}
Note that we have $N_1 = 1$.
By the definition of $x_1^k$ in \eqnok{def_x_k_t_sddp1} and  our assumption
in \eqnok{gap_bnd_ass1_sddp}, we have
\begin{align}
&\bbe[H_1(x^k_1, \tilde c_{t1}) + \lambda V_2(x^k_1)
-  \underline F_{11}^{k-1} (x_1^k)  | {\cal I}_{k-1,T}] \nn \\
&= \bbe[H_1(x^k_1, \tilde c_{t1}) + \lambda V_2(x^k_1) | {\cal I}_{k-1,T} ]
-\bbe [H_1(x^k_1, \tilde c_{t1}) + \lambda \underline V_{2}^{k-1} (x_1^k) | {\cal I}_{k-1,T}] \nn \\
&=  \lambda \bbe[V_2(x^k_1) - \underline V_{2}^{k-1} (x_1^k)  | {\cal I}_{k-1,T}] \nn \\
&\le  \lambda \epsilon_1. \label{eq:hh_1_sddp}
\end{align}
Now consider the $t$-th stage for any $t \ge 2$.
By the definition of $x_t^k$ in \eqnok{def_x_k_t_sddp1}, we have
\begin{align*}
H_t(x_t^k, \tilde c_{t i_t}) + \lambda \underline V_{t+1}^{k-1}(x_t^k)
&= \min \{ H_t(x, \tilde c_{t i_t}) + \lambda \underline V_{t+1}^{k-1}(x): x \in X_t(x_{t-1}^k)\} \\
&\le \min\{ H_t(x, \tilde c_{t i_t}) + \lambda  V_{t+1}(x): x \in X_t(x_{t-1}^k)\}\\
&= \nu_{t i_t} (x^k_{t-1})
\end{align*}
for any $t \ge 2$.
Taking conditional expectation on both sides of the above inequality and using 
our assumption 
$
\lambda \bbe[V_{t+1}(x_t^k) - \underline V_{t+1}^{k-1}(x_t^k)  | {\cal I}_{k,t-1}] \le \epsilon_{t-1},
$
we then have
\begin{align*}
\bbe[H_t(x_t^k, \tilde c_{t i_t}) + \lambda V_{t+1}(x_t^k) | {\cal I}_{k,t-1}] &\le \bbe[\nu_{t i_t} (x^k_{t-1}) | {\cal I}_{k,t-1}] +  \epsilon_{t-1} \\
&=\bbe[V_{t}(x^k_{t-1}) | {\cal I}_{k,t-1}] + \epsilon_{t-1}\\
&= \bbe[V_{t}(x^k_{t-1})| {\cal I}_{k,t-2}] + \epsilon_{t-1},
\end{align*}
where the first identity follows from the definition of $V_t$ and the selection of $i_t$,
and the second identity follows from the fact that $x^k_{t-1}$ is independent of $i_t^k$.
Multiplying $\lambda^{t-1}$ to both side of the above inequalities,
summing them up with the inequalities in \eqnok{eq:hh_1_sddp}, and using the fact that $V_{T+1}(x_T^k) =0$, 
we have
\begin{align*}
\tsum_{t=1}^T \lambda^{t-1} \bbe[H_t(x_t^k, \tilde c_{t i_t}) | {\cal I}_{k,t-1}]  -  \bbe[ \underline F_{11}^{k-1} (x_1^k)|  {\cal I}_{k-1,T}]
\le \tsum_{t=1}^{T} \lambda^{t-1} \epsilon_{t-1}.
\end{align*}
\end{proof}

We also need to use the following well-known result
for the martingale difference sequence when establishing the iteration complexity of SDDP. 

\begin{lemma}\label{wellknown-sco}
Let $\xi_{[t]} \equiv \{\xi_1,\xi_2,\ldots, \xi_t\}$  be a sequence of iid random variables, and
$\zeta_t=\zeta_t(\xi_{[t]})$ be deterministic Borel functions of
$\xi_{[t]}$ such that $\bbe_{|\xi_{[t-1]}}[\zeta_t]=0$ a.s. and
$\bbe_{|\xi_{[t-1]}}[\exp\{\zeta_t^2/\sigma_t^2\}]\leq\exp\{1\}$ a.s.,
where $\sigma_t>0$ are deterministic. Then
\beq \label{eq:martingale1}
\forall\lambda\geq0:
\Prob\left\{\tsum_{t=1}^N\zeta_t>\lambda\sqrt{\tsum_{t=1}^N\sigma_t^2}\right\}
\leq\exp\{-\lambda^2/3\}.
\eeq
and
\beq \label{eq:martingale2}
\forall\lambda\geq0:
\Prob\left\{\tsum_{t=1}^N\zeta_t< -\lambda\sqrt{\tsum_{t=1}^N\sigma_t^2}\right\}
\leq\exp\{-\lambda^2/3\}.
\eeq
\end{lemma}

\begin{proof}
The proof of  \eqnok{eq:martingale1} can be found, e.g.,
Lemma 2 in \cite{lns11}. In addition, \eqnok{eq:martingale2} 
follows from \eqnok{eq:martingale1}  by replacing $\zeta_t$ with $-\zeta_t$.
\end{proof}

\vgap

We are now ready to establish the complexity of SDDP.

\begin{theorem} \label{the_main_sddp_rand}
Suppose that the norm used to define the bound 
$D_t$ in \eqnok{bound_X_sddp} is the $l_\infty$ norm.
Also assume that $\delta_t \in [0, +\infty)$ and
$\epsilon_t$ are defined in \eqnok{eq:closeenough_dcp}.
Let $K$ denote the number of iterations performed by SDDP before it
finds a forward path $(x_1^k, \ldots, x_T^k)$ defined in \eqnok{def_x_k_t_sddp1} for  problem \eqnok{multstage1_SAA} s.t.
\begin{align}
F_{11}(x_1^k) - F^* &\le \epsilon_0, \label{teminations_sddp1} \\
\tsum_{t=1}^T \lambda^{t-1} \bbe[H_t(x_t^k, \tilde c_{t i_t}) | {\cal I}_{k,t-1}]  -  \bbe[ \underline F_{11}^{k-1} (x_1^k)|  {\cal I}_{k-1,T}] 
&\le \tsum_{t=1}^T \lambda^{t-1} \epsilon_{t-1}. \label{teminations_sddp2}
\end{align}
Then we have $\bbe[K] \le  \bar K \bar N  + 2$, where $\bar K$ and
$\bar N$ are defined in \eqnok{def_N_sddp} and \eqnok{def_max_N_t}, respectively. 
In addition, for any $\alpha \ge 1$, we have
\begin{align}
\prob\{K \ge \alpha \bar K  \bar N + 1\} \le  \exp\left(-\tfrac{ (\alpha-1)^2 \bar K^2}{2 \alpha \bar N} \right).
\end{align}
\end{theorem}

\begin{proof}
First note that if $ \tilde g_{t}^k \le \delta_{t}$ for all $t = 1, \ldots, T-1$,
then \eqnok{teminations_sddp1} and \eqnok{teminations_sddp2}
must hold in view of the discussions after Lemma~\ref{lemma_prob_saturation} (c.f. \eqnok{gap_bnd_ass1_sddp0} and \eqnok{gap_bnd_ass1_sddp})
and Lemma~\ref{lem:bnd_gap_sddp}.
Therefore, the event $\tilde g_{t}^k \le \delta_{t}$ for all $t = 1, \ldots, T-1$
will not happen for any $1 \le k \le K-1$. In other words, we have
$\prob\{ \tilde g_{t}^k \le \delta_{t}, t=1, \ldots, T-1\} = 0$ for all $1 \le k \le K-1$, which, in view of \eqnok{bnd_iter_prob_0}, implies that
for any $1 \le k \le K-1$, 
\beq \label{bnd_iter_prob}
\prob\{\pIter^k=1\}  \ge \tfrac{1}{\bar N}.
\eeq
Moreover, observe that we must have
\beq \label{bbb_KKK}
\tsum_{k=1}^{K-2} \pIter^k \le \bar K,
\eeq
since otherwise the algorithm has generated totally $\bar K$ $\epsilon_t$-saturated and $\delta_t$-distinguishable
search points during the first $K-2$ iterations, and thus must terminate at the $K-1$ iterations (i.e.,
\eqnok{teminations_sddp1} and \eqnok{teminations_sddp2} must hold due to $\tilde g_{t}^{K-1} \le \delta_{t}$ for all $t = 1, \ldots, T-1$).
%
Taking expectation on both sides of \eqnok{bbb_KKK}, we have
\[
\bar K \ge \bbe_K[\bbe[\tsum_{k=1}^{K-2} \pIter^k | K] ] \ge \bbe_K[ \tfrac{K-2}{\bar N}] = \tfrac{\bbe{[K]-2}}{\bar N},
\]
implying that $\bbe[K] \le  \bar N \bar K + 2$.

Now we need to bound the probability that
the algorithm does not terminate in $ \alpha \bar N \bar K + 1 $ iterations for $\alpha \ge 1$.
Observe that
\begin{align}
\prob\{K \ge \alpha \bar N \bar K + 1\} \le \prob\{ \tsum_{k=1}^{\alpha \bar N \bar K} \pIter^k < \bar K\}, \label{prob_bnd_temp1}
\end{align}
since $K \ge \alpha \bar N \bar K + 1$ must imply that $\tsum_{k=1}^{\alpha \bar N \bar K} \pIter^k < \bar K$.
Note that
$ \pIter^k - \bbe[\pIter^k]$ is a margingale-difference sequence,
and $\bbe[\exp((\pIter^k)^2)] \le 1$. Hence we have
\begin{align}
&\prob\{ \tsum_{k=1}^{\alpha \bar N \bar K} \pIter^k < \alpha \bar K - \lambda \sqrt{\alpha \bar N \bar K}\} \nn\\
&\le \prob\{ \tsum_{k=1}^{\alpha \bar N \bar K} \pIter^k \le \tsum_{k=1}^{\alpha \bar N \bar K} \bbe[\pIter^k]
- \lambda \sqrt{ \alpha \bar N \bar K}\} \nn\\
&\le \exp(-\lambda^2/2), \forall \lambda > 0,
\end{align}
where the first inequality follows from the fact that $ \bbe[\pIter^k] \ge 1/ \bar N$, $k = 1, \ldots, \alpha \bar N \bar K$, 
and thus $\tsum_{k=1}^{\alpha \bar N \bar K} \bbe[\pIter^k]\ge \alpha \bar K$,
and the second inequality follows from Lemma~\ref{wellknown-sco}.
Setting 
\[
\lambda = \tfrac{(\alpha-1) \bar K}{\sqrt{\alpha \bar N}}
\]
in the above relation, we then conclude that
\beq \label{prob_bnd_temp2}
\prob\{ \tsum_{k=1}^{\alpha \bar N \bar K} \pIter^k < \bar K\}
\le \exp\left(-\tfrac{ (\alpha-1)^2 \bar K^2}{2 \alpha \bar N} \right).
\eeq
Combining \eqnok{prob_bnd_temp1} and \eqnok{prob_bnd_temp2},
we then conclude that
\[\prob\{K \ge \alpha \bar N \bar K + 1\} \le  \exp\left(-\tfrac{ (\alpha-1)^2 \bar K^2}{2 \alpha \bar N} \right), \, \forall \alpha \ge 1.
\]
\end{proof}

We have the following immediate consequence of Theorem~\ref{the_main_sddp_rand}.

\begin{corollary} \label{cor_main_sddp_rand}
Suppose that $n_t \le n$, $D_t \le D$, $\max\{M_t,\underline M_t\} \le M$ and $\delta_t = \epsilon$ for all $t=1, \ldots, T$.
Let $K$ denote the number of iterations performed by the SDDP method before it
finds a forward path $(x_1^k, \ldots, x_T^k)$ of problem \eqnok{multstage1_SAA} s.t.
\begin{align}
&F_{11}(x_1^k) - F^* \le 2 M  \min\{\tfrac{1}{1-\lambda}, T-1 \} \, \epsilon, \label{teminations_sddp_rand1}\\
&
\tsum_{t=1}^T \lambda^{t-1} \bbe[H_t(x_t^k, \tilde c_{t i_t}) | {\cal I}_{k,t-1}]  -  \bbe[ \underline F_{11}^{k-1} (x_1^k)|  {\cal I}_{k-1,T}]  \nn\\
&\quad \quad \quad \le 2 M  \min\{ \tfrac{1 }{(1-\lambda)^2}, \tfrac{T(T-1)}{2}\} \, \epsilon.
 \label{teminations_sddp_rand2}
\end{align}
Then we have $\bbe[K] \le  \bar K_\epsilon \bar N  + 2$, where $\bar K_\epsilon$ and
$\bar N$ is defined in \eqnok{def_N_s_sddp} and \eqnok{def_max_N_t}, respectively. 
In addition, for any $\alpha \ge 1$, we have
\[
\prob\{K \ge \alpha \bar K_\epsilon  \bar N + 1\} \le  \exp\left(-\tfrac{ (\alpha-1)^2 \bar K_\epsilon^2}{2 \alpha \bar N} \right).
\]
\end{corollary}

\begin{proof}
The relations in \eqnok{teminations_sddp_rand1} and \eqnok{teminations_sddp_rand2}
follow by using the bound \eqnok{def_epsilon_t_precisely} for $\epsilon_0$ in \eqnok{teminations_sddp1}
and by using the bound \eqnok{def_epsilon_t_precisely1} for $\tsum_{t=1}^T \epsilon_{t-1}$ in \eqnok{teminations_sddp2}, respectively.
Moreover, the bounds on $\bbe[K]$ and $\prob\{K \ge \alpha \bar K_\epsilon  \bar N+1\}$
directly follows from Theorem~\ref{the_main_sddp_rand} by replacing $\bar K$ with $\bar K_\epsilon$.
\end{proof}


\vgap

We now add a few remarks about the results obtained in Theorem~\ref{the_main_sddp_rand}
and Corollary~\ref{cor_main_sddp_rand}. Firstly, since SDDP is a randomized
algorithm, we provide bounds on the expected number of iterations required to find an approximate
solution of problem~\eqnok{multstage1_SAA}. We also show that the probability of having
large deviations
from these expected bounds for SDDP decays exponentially fast.
Secondly, the complexity bounds for
the SDDP method is $\bar N$ times worse than those in Theorem~\ref{the_main_EDDP} for the EDDP method,
even though the dependence on other parameters, including
$n$ and $\epsilon$, remains the same.
Thirdly, similar to DDP and EDDP, the complexity of SDDP actually  
depends the dimension of the effective feasible region $\bar {\cal X}_t$ in \eqnok{def_N_sddp},
which can be smaller than $n_t$.

\begin{remark}
It should be noted that although the complexity of SDDP is worse than those for DDP and EDDP,
its performance in earlier phase of the algorithm should be similar to that of DDP.
Intuitively, for earlier iterations, the tolerance parameter $\delta_t$ are large.
As long as $\delta_t$ are large enough so that the solutions $\tilde x^k_{t i}$ are contained
within a ball with diameter roughly in the order of $\delta_t$, one can choose any
point randomly from $\tilde x^k_{t i}$ as $x^k_t$. In this case, SDDP will perform similarly to DDP and EDDP. This may explain why 
SDDP exhibits good practical performance for low accuracy region.
For high accuracy region, the new EDDP algorithm 
seems to be a much better choice in terms of its theoretical complexity. 
In practice, it might make sense to run SDDP in earlier phases (due to
its simplicity), and then switch to EDDP to achieve higher accuracy.
\end{remark}

\vgap

As shown in Theorem~\ref{the_main_sddp_rand}
and Corollary~\ref{cor_main_sddp_rand}, we can show the convergence of the gap between a stochastic upper
bound on $F_{11}(x_1^k)$, given by $\tsum_{t=1}^T \lambda^{t-1} H_t(x_t^k, \tilde c_{t i_t})$, 
and the lower bound $\underline F_{11}^{k-1} (x_1^k)$,
 generated by the SDDP method. In order to obtain a statistically more reliable
 upper bound, we can run the forward phase $L \ge 1$ times in each iteration. In particular, 
 we can replace the forward phase in Algorithm~\ref{algo_sddp1}
 with the one shown in Algorithm~\ref{algo_sddp_estimated}.
 We can then compute the average and  estimated standard deviation of $\ub_k$ over these $L$ runs of
 the forward phase.
 
 \begin{algorithm}[H] 
\caption{Forward phase with upper bound estimation}
\begin{algorithmic}[1]

\For {$l = 1, \ldots, L$} \Comment{Forward phase.} 

\State Set $\tilde F_l = 0$.

\For {$t = 1, \ldots, T$}  

\State Pick up $i_t$ from $\{1, 2, \ldots, N_t\}$ uniformly randomly.

\State Set $x_{t}^k$ according to \eqnok{def_x_k_t_sddp1} and 
$
\tilde F_l = \tilde F_l + \lambda^{t-1} H_t(x_t^k, \tilde c_{t{i_t}}).
$
\EndFor

\State Set $\ub_k = \ub_k + \tilde F_l$.

\EndFor

\State Set $\ub_k = \ub_k / L$.

\end{algorithmic} \label{algo_sddp_estimated}
\end{algorithm}
 
 It should be noted, however, that the convergence of the SDDP method
 only requires $L = 1$. 
 To choose $L > 1$ helps  
 to properly terminate the algorithm by
 providing a statistically more accurate upper bound. 
 Moreover, since each run of the forward phase will generate a
 forward path, we can use these $L$ forward paths to run the backward phases
 in parallel to accelerate the convergence of SDDP.
 Following a similar analysis to the basic version of SDDP,
 we can show that
 the number of iterations required by the above 
 variant of SDDP will be $L$ times smaller than the one
 for Algorithm~\ref{algo_sddp1}, but each iteration
 is computationally more expensive or requires
 more computing resources for parallel processing.

\section{Conclusion} \label{sec-remark}
In this paper, we establish the complexity of a few 
cutting plane algorithms, including DDP, EDDP and SDDP, for solving dynamic convex optimization problems.
These methods build up piecewise linear functions
to approximate the value functions through the backward phase and
generate feasible policies in the forward phase by utilizing these
cutting plane models. 
For the first time in the literature, we establish the total number of iterations
required to run these forward and backward phases in order
to compute a certain accurate solution. Our results reveal that these
methods have a mild dependence on the number of stages $T$.

It is worth noting that in our current analysis we assume that all the subproblems in
the forward and backward phases are solved exactly. However, we
can possibly extend the basic analysis to the case when these subproblems are solved
inexactly as long as the errors are small enough. Moreover,
we did not make any assumptions
on how the subproblems are solved. As a result, it is possible to
extend our complexity results to multi-stage stochastic binary (or integer)
programming problems (see, e.g., \cite{ZouAhmedSun19-1}).
In addition, the major analysis for SDDP presented in this paper
does not rely on the convexity, but the Lipschitz
continuity of the value functions and their lower approximations. Hence, it seems to be possible 
to adapt our analysis for SDDP-type methods with nonconvex approximations for the value functions~\cite{MIDAS2016,AhemdCabralCosta19}.

We have discussed a few different ways to terminate
 DDP, EDDP and SDDP. More specifically,
DDP can be terminated by calculating
the gap between the upper and lower bounds, and EDDP
is a variant of SDDP with rigorous termination based on
the saturation of search points, whereas SDDP is usually terminated
by resorting to statistically valid upper bounds coupled with
the lower bounds obtained from the cutting plane models.
Recently an important line of research has been developed to
design SDDP-like methods with more reliable and efficient termination criterions (see, e.g.,
\cite{Georghiou19-1,Baucke17-1,Vincent2020}).
It will be interesting to study the complexity of these new methods
in the future. 

\renewcommand\refname{Reference}

\bibliographystyle{plain}
\bibliography{glan-bib}

\end{document}